\documentclass{amsart}

\RequirePackage{doi}
\usepackage{hyperref}
\usepackage{blindtext}

\usepackage{amsmath}
\usepackage{tikz-cd}
\usepackage{amsthm}
\usepackage{hyperref}
\usepackage{amsfonts,graphics,amsthm,amsfonts,amscd,latexsym}
\usepackage{epsfig}
\usepackage{flafter}
\usepackage{mathtools}
\usepackage{comment}
\usepackage{stmaryrd}
\usepackage{cite}

\usepackage{amssymb}

\hypersetup{
    colorlinks=true,    
    linkcolor=blue,          
    citecolor=blue,      
    filecolor=blue,      
    urlcolor=blue           
}
\usepackage{tikz}
\usetikzlibrary{graphs,positioning,arrows,shapes.misc,decorations.pathmorphing}

\tikzset{
    >=stealth,
    every picture/.style={thick},
    graphs/every graph/.style={empty nodes},
}

\tikzstyle{vertex}=[
    draw,
    circle,
    fill=black,
    inner sep=1pt,
    minimum width=5pt,
]
\usepackage[position=top]{subfig}
\usepackage{amssymb}
\usepackage{color}

\calclayout

\usetikzlibrary{decorations.pathmorphing}
\tikzstyle{printersafe}=[decoration={snake,amplitude=0pt}]

\newcommand{\supp}{\operatorname{supp}}

\newcommand{\pp}{\mathbb{P}}

\newcommand{\qq}{\mathbb{Q}}
\newcommand{\zz}{\mathbb{Z}}

\newcommand{\rr}{\mathbb{R}}

\newcommand{\kk}{\mathbb{K}}

\def\O#1.{\mathcal {O}_{#1}}			
\def\pr #1.{\mathbb P^{#1}}				
\def\af #1.{\mathbb A^{#1}}			
\def\ses#1.#2.#3.{0\to #1\to #2\to #3 \to 0}	
\def\xrar#1.{\xrightarrow{#1}}			
\def\K#1.{K_{#1}}						
\def\bA#1.{\mathbf{A}_{#1}}			
\def\bM#1.{\mathbf{M}_{#1}}				
\def\bL#1.{\mathbf{L}_{#1}}				
\def\bB#1.{\mathbf{B}_{#1}}				
\def\bK#1.{\mathbf{K}_{#1}}			
\def\subs#1.{_{#1}}					
\def\sups#1.{^{#1}}

\usepackage{tikz}
\usetikzlibrary{matrix,arrows,decorations.pathmorphing}

  \newtheorem{introthm}{Theorem}

  \newtheorem{introcor}{Corollary}

  \newtheorem{theorem}{Theorem}[section]
  \newtheorem{lemma}[theorem]{Lemma}
  
  \newtheorem{corollary}[theorem]{Corollary}

  \newtheorem{definition}[theorem]{Definition}

\theoremstyle{remark}

\numberwithin{equation}{section}

\usepackage[all]{xy}

\begin{document}

\title[Minimal log discrepancies of regularity one]
{Minimal log discrepancies of regularity one}

\author[J.~Moraga]{Joaqu\'in Moraga}
\address{Department of Mathematics, Princeton University, Fine Hall, Washington Road, Princeton, NJ 08544-1000, USA
}
\email{jmoraga@princeton.edu}

\maketitle

\begin{abstract}
In this article, we use the cone of nef curves to study minimal log discrepancies.
The first result is an improvement of the nef cone theorem 
in the case of log Calabi-Yau dlt pairs.
Then, we prove that the ascending chain condition for
$n$-dimensional minimal log discrepancies of regularity one holds around zero.
Furthermore, we show that there exists an upper bound for the minimal log discrepancy of any $n$-dimensional klt singularity of regularity one.
\end{abstract}

\setcounter{tocdepth}{1} 
\tableofcontents

\setcounter{tocdepth}{1}
\tableofcontents

\section{Introduction}

The minimal log discrepancy is an invariant of klt singularities.
Regardless of the simplicity of its definition, 
proving statements about it is often challenging. 
Shokurov proposed a conjecture regarding the ascending chain condition of minimal log discrepancies~\cite{Sho04}.
On the other hand, 
Ambro conjectured that the minimal log discrepancy is lower semicontinuous~\cite{Amb99}
These conjectures imply the termination of flips~\cite{Sho04}.
Many results about termination of flips rely on 
theorems which are special cases of these two conjectures (see, e.g.,~\cite{Bir07,Mor18a,HM20}).
It is expected that the boundedness of Fano varieties in dimension $n$ implies the ascending condition for minimal log discrepancies in dimension $n+1$ (see, e.g.,~\cite{BS10}). The underlying principle is that $n$-dimensional Fano type varieties are the building blocks of partial resolutions of $(n+1)$-dimensional klt singularities.
More precisely, if the minimal log discrepancy of the klt singularity is bounded away from zero,
we expect a similar behavior on the Fano type varieties extracted on a partial resolution.
Thus, it is natural to apply Birkar's boundedness of Fano varieties in this context.
This approach works especially well in the case of singularities of regularity zero, the so-called exceptional singularities~\cite{Mor18b,HLS19}.
These singularities are deformations of cones over exceptional Fano type varieties~\cite{Mor18c,HLM20}.

Given a $n$-dimensional exceptional klt singularity $(X;x)$, we can perform a blow-up $\pi\colon Y\rightarrow X$ that extracts a unique prime exceptional divisor $E$. The variety $E$ is an exceptional $(n-1)$-dimensional Fano type variety.
These varieties belong to a bounded family~\cite{Bir19}.
In~\cite{Mor18b}, we use this boundedness result to prove the existence of a curve in the smooth locus of $E$
with bounded degree with respect to $-K_E$.
A straightforward computation implies that the log discrepancy of $(X;x)$ at $E$ is bounded above by a constant only depending on the dimension. 
In the general setting of $n$-dimensional klt singularities,
it is not possible to extract a bounded $(n-1)$-dimensional Fano type variety over the singularity.
In order to generalize the approach of~\cite{Mor18b},
we need to consider partial resolutions of the singularity
extracting several exceptional prime divisors. 
Furthermore, we need to show that each such exceptional divisor
contains a special kind of curve on its smooth locus.
More generally, we need to understand whether we can generate the cone of curves of a Fano type variety with \textit{nice} curves.
For instance, curves which are contained in the smooth locus and have bounded degree with respect to the anti-canonical divisor.
In this direction, it is natural to look at those curves which come from fibers of Mori fiber space structures of the Fano type varieties.
Hence, we expect to find many of these curves in the $K_X$-negative domain of the cone of movable curves.
The classic cone theorem for algebraic varieties
has an analog in the case of nef curves (see, e.g.,~\cite{Bat89,Ara10,Leh12}).
Our first theorem is an enhancement of these theorems
in the case of log Calabi-Yau pairs.
The main new outcome is that we can control certain numerical invariants
of the $K_X$-extremal negative curves.

\begin{introthm}\label{introthm:nef-cone}
Let $n$ and $N$ be positive integers.
There exists $k:=k(n,N)$, only depending on $n$ and $N$,
satisfying the following.
Let $(X,B)$ be a $n$-dimensional dlt log Calabi-Yau pair
such that $N(K_X+B)\sim 0$.
Let $E_1,\dots,E_r$ be the prime components of $\lfloor B\rfloor$. 
Let $B'=B-\lfloor B\rfloor$.
There are countable many $(K_X+B')$-negative curves $C_i$, with $i\in \zz$, satisfying the following conditions: 
\begin{enumerate}
    \item the curve $C_i$ is movable and lies in the smooth locus of $X$,
    \item  the curve $C_i$ is either disjoint from $B'$ or intersect it transversally with $B'\cdot C_i \leq k$, 
    \item for every $j\in \{1,\dots,r\}$ the curve $C_i$ is either disjoint from $E_j$ or intersect it transversally with $E_j \cdot C_i \leq k$,
    \item for each subset $Z\subset X$ of codimension at least two
    and $x\in X$ general, we can find $C_i'\equiv C_i$ with $C_i'\cap Z=\emptyset$ and $x\in C_i'$, and 
    \item $C_i$ intersects at least one and at most $n+1$ of the $E_j$'s.
\end{enumerate}
Moreover, the following equality holds: 
\begin{equation}\label{eq:nef-cone} 
\overline{NE}_1(X)_{K_X+B' \geq 0} +
\overline{NM}_1(X) = 
\overline{NE}_1(X)_{K_X+B' \geq 0} +
\overline{\sum_{i\in \zz}\rr_{\geq 0}[C_i]}.
\end{equation} 
Furthermore, the rays $\rr_{\geq 0}[C_i]$ only accumulates to hyperplanes that support both
$\overline{NM}_1(X)$ and the cone
$\overline{NE}_1(X)_{K_X+B' \geq 0}$.
\end{introthm}

We stress that in the context of Theorem~\ref{introthm:nef-cone}, we crucially use the dlt condition.
Indeed, the dlt condition and the control of the index will allow us
to control the singularities of $X$. Then, we can apply Birkar's boundedness of Fano varieties 
in the general fiber of a birational Mori fiber space structure of $X$.
This boundedness result will let us deduce (1)-(4) in the statement of Theorem~\ref{introthm:nef-cone}.
The upper bound on the number of prime components of $\lfloor B\rfloor$ that $C_i$ can intersect 
is a consequence of a computation using the complexity of log pairs (see, e.g.,~\cite{Kol92,BMSZ18,RS21}).

It is intuitive to ask how many, if any, curves $C_i$ do we need
to obtain the equality~\eqref{eq:nef-cone}.
In principle, the more divisorial log canonical centers we have, 
the more curves $C_i$ we expect to need.
In the case that $\lfloor B\rfloor$ is disconnected, 
we know that it must have at most two components~\cite[Theorem 1.2]{Bir21}.
Furthermore, if $\lfloor B\rfloor$ has exactly two components,
then the pair $(X,B)$ is birational to a $\pp^1$-link (see Definition~\ref{def:p1-links}).
The next theorem asserts that such $\pp^1$-link structure is unique.
This means that we can find a unique canonically-negative extremal movable curve in the cone~\eqref{eq:nef-cone}.

\begin{introthm}\label{introthm:uniqueness} 
Let $(X,B)$ be a dlt log Calabi-Yau pair.
Assume that the set of log canonical centers of $(X,B)$ is disconnected.
Let $B'=B-\lfloor B\rfloor$.
Then, there exists a unique $(K_X+B')$-negative extremal nef curve $C_0$
for which 
\[
\overline{NE}_1(X)_{K_X+B' \geq 0} +
\overline{NM}_1(X) = 
\overline{NE}_1(X)_{K_X+B' \geq 0} +
\rr_{\geq 0}[C_0].
\]
Furthermore, the following intersection properties hold:
\[
B'\cdot C_0=0,\quad 
E_1\cdot C_0=1,\quad 
E_2\cdot C_0=1,\text{ and }\quad 
K_X\cdot C_0=-2.
\]
In particular, if $-(K_X+B')$ is ample, then 
$X\simeq \pp^1$, $B'=0$, and $B=\{0\}+\{\infty\}$.
\end{introthm}

The regularity of a klt singularity is the largest dimension among the dual complexes of lc complements.
If the regularity is equal to the dimension minus $1$,
then we expect these singularities to behave similarly to toric singularities~\cite{Mor20b}.
On the other hand, singularities of regularity zero are exceptional~\cite{Mor18b}.
There has been some recent development in the topology of klt singularities~\cite{LLM19,Bra20,BFMS20,BM21,Mor21}.
In particular, for a $n$-dimensional klt singularity $(X;x)$, we know that its regional fundamental group is almost a finite abelian group of rank at most ${\rm reg}(X;x)$.
Thus, the regularity controls to a large extent the topology of the singularity.
We also expect the regularity to be a central invariant in the study of minimal log discrepancies.

In the case that $(X;x)$ is a $n$-dimensional klt singularity of regularity one,
we can choose a $N$-complement $(X,B)$ so that $\mathcal{D}(X,B)$ is either a circle or a closed interval (see Definition~\ref{def:complement}).
Taking a dlt modification $(Y,B_Y)$ of $(X,B;x)$
and performing adjunction to each of the exceptional divisors, 
we arise to the situation of Theorem~\ref{introthm:nef-cone}. 
In this case, most of the exceptional divisors of $\pi\colon Y\rightarrow X$ are birational to $\pp^1$-links, i.e., they are models as in Theorem~\ref{introthm:uniqueness}.
If they exist, the exceptional divisors corresponding to the boundary of the dual complex $\mathcal{D}(Y,B_Y)$ will be plt log Calabi-Yau pairs
with a unique log canonical center.
Hence, we can find some nice curves in these exceptional divisors by applying Theorem~\ref{introthm:nef-cone} and Theorem~\ref{introthm:uniqueness}. 
Thus, there is a natural curve to choose corresponding to each vertex of the dual complex. Then,
we can write a linear system of equations that 
allows us to compute each $a_{E_i}(X)$.
The value $a_{E_i}(X)$ will be expressed
in terms of the intersection of the aforementioned curves with the 
components of $B_Y$.
Nevertheless, this linear system of equations is not entirely trivial. 
For instance, the number of variables can go up as we consider different klt singularities of regularity one.
To remedy this, we will construct a surface toric singularity whose resolution of singularities contains the same combinatorial data (i.e., the same linear system) as that of $(Y,B_Y)$. 
This will allow us to conclude the following theorem.

\begin{introthm}\label{introthm:ACC}
Let $n$ be a positive integer.
There exists a constant $N:=N(n)$,
only depending on $n$, which satisfies the following.
Let $\mathcal{M}_{n,1}$ be the set of
minimal log discrepancies 
of $n$-dimensional $\qq$-factorial klt singularities of regularity one.
Then, the set
\[
\mathcal{M}_{n,1}\cap \left(0,\frac{1}{N} \right)
\]
satisfies the ascending chain condition.
\end{introthm}

Among all exceptional divisors over a klt germ $(X;x)$ those that compute a log canonical place of an lc pair $(X,B;x)$ play a central role in the proof of the previous theorem.
These divisors enjoy many good properties. They carry the structure of log Calabi-Yau pairs~\cite{Hac14}.
Furthermore, they can be extracted by a projective birational morphism in such a way that they are the unique exceptional divisor~\cite{Mor20}.
In order to prove Theorem~\ref{introthm:ACC}, we will show that the minimum among log discrepancies at divisors that compute log canonical places satisfies the ascending chain condition.
In particular, for every klt singularity of regularity one, 
we will prove the existence of such a divisor with log discrepancy bounded above by a constant only depending on the dimension.

\begin{introcor}\label{introcor:bounded-ext}
Let $n$ be a positive integer.
There exists a constant $a(n)$, only depending on $n$, which satisfies the following.
Let $(X;x)$ be a $n$-dimensional $\qq$-factorial klt singularity of regularity one.
Then, there exists a projective birational morphism 
$\pi\colon Y\rightarrow X$, with purely divisorial exceptional loci,
which extracts a unique divisor $E$ with $a_E(X;x)\leq a(n)$.
\end{introcor}

Once the ascending chain condition is established, 
we can study the accumulation points of the set of minimal log discrepancies.
In the case of exceptional singularities~\cite{Mor18b}, the only accumulation point is zero.
In the case of singularities of regularity one, we prove that the accumulation points of the set of minimal log discrepancies only accumulate to zero.

\begin{introcor}\label{introcor:accum}
Let $n$ be a positive integer.
There exists a constant $N:=N(n)$,
only depending on $n$,
which satisfies the following.
Let $\mathcal{M}_{n,1}$ be the set of minimal log discrepancies of $n$-dimensional $\qq$-factorial klt singularities of regularity one.
Then, the accumulation points of the set
\[ 
{\rm Acc}\left( 
\mathcal{M}_{n,1} \cap \left(0,\frac{1}{N}\right) 
\right) 
\]
only accumulate to zero.
\end{introcor}

We expect the circle of ideas introduced in this paper to help to prove the corresponding statements in the general case of klt singularities. 
This will be considered in forthcoming papers by the author.

\subsection*{Acknowledgements} 
The author would like to thank Brian Lehmann, Mihai Fulger, and 
Vyacheslav Shokurov, for many useful comments. 

\section{Preliminaries}

In this section,
we recall some preliminaries 
that will be used in this article:
singularities of the minimal model program,
theory of complements,
dual complexes, 
regularity, and toric surface singularities.
Throughout this article, we work over an algebraically closed field $\kk$ of characteristic zero.
All the considered varieties are normal and quasi-projective unless otherwise stated.

\subsection{Singularities and complements}

In this subsection, we recall the basic notions of the singularities of the minimal model program and the definition of complements (see, e.g.,~\cite{KM98,Kol13,Bir20}).

\begin{definition}
{\em 
Let $f\colon Y\rightarrow X$ be a morphism.
We say that $f$ is a contraction if $f_*\mathcal{O}_Y=\mathcal{O}_X$.
In particular, if $Y$ is normal and $f$ is a contraction, then $X$ is normal.
A {\em fibration} is a contraction with positive dimensional general fiber.
}
\end{definition}

\begin{definition}
{\em 
A {\em sub-pair} is a couple $(X,\Delta)$ where $X$ is a normal quasi-projective variety and $\Delta$ is a $\qq$-divisor such that $K_X+\Delta$ is a $\qq$-Cartier $\qq$-divisor.
A {\em pair} (or {\em log pair}) is a sub-pair with $\Delta\geq 0$.
We may write $(X,\Delta;x)$ to denote a sub-pair with a base point.
}
\end{definition}

\begin{definition}
{\em
Let $X$ be a normal quasi-projective variety.
A {\em prime divisor over $X$} is a prime divisor which lies in a normal variety admitting a projective birational morphism to $X$.
This means that we can find a projective birational morphism $\pi\colon Y\rightarrow X$ so that $E\subset Y$ is a prime divisor.
The center of $E$ on $X$ is the image of $E$ on $X$ and is denoted by $c_X(E)$.

Let $(X,\Delta)$ be a log pair.
Let $E$ be a prime divisor over $X$.
The {\em log discrepancy} of $(X,\Delta)$ at $E$ is defined to be 
\[
a_E(X,\Delta):=
1-{\rm coeff}_E(K_Y-\pi^*(K_X+\Delta)).
\]
Here, as usual, we pick $K_Y$ so that
$\pi_* K_Y=K_X$.

Let $(X,\Delta)$ be a log pair.
A {\em log resolution} of $(X,\Delta)$ is a projective birational morphism $\pi\colon Y\rightarrow X$ satisfying the following conditions:
\begin{enumerate}
    \item $Y$ is a smooth variety,
    \item the exceptional locus of $\pi$ is purely divisorial, and
    \item $\pi^{-1}_*\Delta + {\rm Ex}(\pi)_{\rm red}$ is a reduced divisor with simple normal crossing support.
\end{enumerate}
By Hironaka's resolution of singularities,
we know that any log pair admits a log resolution.
}
\end{definition}

\begin{definition}
{\em 
Let $(X,\Delta)$ be a log pair.
We say that $(X,\Delta)$ is {\em Kawamata log terminal} (or {\em klt} for short) if all its log discrepancies are positive.
This means that $a_E(X,\Delta)>0$ for every prime divisor $E$ over $X$.
It is known that a log pair $(X,\Delta)$ is klt if and only if all 
the log discrepancies corresponding 
to prime divisors on a log resolution of $(X,\Delta)$ are positive.

Analogously, we say that a log pair $(X,\Delta)$ is {\em log canonical} (or {\em lc} for short) if all its log discrepancies are non-negative, i.e., 
$a_E(X,\Delta)\geq 0$ for every prime divisor $E$ over $X$.
A pair $(X,\Delta)$ is log canonical if and only if all the log discrepancies corresponding to divisors on a log resolution are non-negative.

A pair $(X,\Delta)$ is said to be {\em $\epsilon$-log canonical} if all its log discrepancies are at least $\epsilon$.
}
\end{definition} 

\begin{definition}
{\em
Let $(X,\Delta;x)$ be a sub-pair.
We define the {\em minimal log discrepancy}
of $(X,\Delta)$ at $x$ to be 
\[
{\rm mld}(X,\Delta;x) :=
\min\{ a_E(X,\Delta) \mid c_E(X)=x\}. 
\]
}
\end{definition}

\begin{definition}
{\em 
Let $(X,\Delta)$ be a log pair.
A {\em log canonical place} of $(X,\Delta)$ is a prime divisor $E$ over $X$ so that $a_E(X,\Delta)=0$.
A {\em non-klt place} of $(X,\Delta)$ is a prime divisor $E$ over $X$ so that 
$a_E(X,\Delta)\leq 0$.
A {\em log canonical center} of $(X,\Delta)$
is the image on $X$
of a log canonical place of $(X,\Delta)$.
A {\em non-klt center} of $(X,\Delta)$
is the image on $X$
of a non-klt place of $(X,\Delta)$.
}
\end{definition}

\begin{definition}
{\em 
A log pair $(X,\Delta)$ is said to be {\em divisorially log terminal} (or {\em dlt})
if there exists an open subset $U\subset X$ which satisfies the following conditions:
\begin{enumerate}
    \item $U$ is smooth and $\Delta|_U$ has simple normal crossing support, 
    \item the coefficients of $\Delta$ are less or equal than one, and
    \item all the non-klt centers of $(X,\Delta)$ intersect $U$ and are given by strata of $\lfloor \Delta \rfloor$.
\end{enumerate}
A pair $(X,\Delta)$ is said to be
{\em purely log terminal} or {\em plt} for short, 
if it is dlt and all its log canonical centers are divisors on $(X,\Delta)$.
}
\end{definition}

The following lemma is well-known
and usually referred to as the existence of $\qq$-factorial dlt modifications (see, e.g.,~\cite[Theorem 3.1]{KK10}).

\begin{lemma}\label{lem:existence-dlt-mod}
{\em 
Let $(X,\Delta)$ be a log canonical pair.
Then, there exists a projective birational morpism $\pi\colon Y\rightarrow X$ so that the following conditions are satisfied:
\begin{enumerate}
    \item $Y$ is $\qq$-factorial,
    \item the exceptional locus of $\pi$ is purely divisorial,
    \item every prime divisor $E$ contracted by $\pi$ satisfies $a_E(X,\Delta)=0$, i.e., $E$ is a log canonical place of $(X,\Delta)$, and 
    \item the pair $(Y,\Delta_Y)$ is dlt where $K_Y+\Delta_Y = \pi^*(K_X+\Delta)$.
\end{enumerate}
}
\end{lemma}

We conclude this subsection by recalling the definitions of complements.

\begin{definition}
\label{def:complement}
{\em 
Let $X\rightarrow Z$ be a projective morphism
and $z\in Z$ be a closed point.
Let $(X,\Delta)$ be a pair.
We say that $B\geq \Delta$ is a {\em $\qq$-complement} of $(X,\Delta)$ over $z\in Z$, if the following conditions are satisfied:
\begin{enumerate}
    \item $(X,B)$ has log canonical singularities, and 
    \item $K_X+B\sim_{\qq,Z} 0$ holds over a neighborhood of $z\in Z$.
\end{enumerate}
In the case that $X\rightarrow Z$ is the identity morphism, then a $\qq$-complement is nothing else than a pair $(X,B)$
with $B\geq \Delta$, which is log canonical around the point $x\in X$.
We say that $B\geq \Delta$ is a {\em $N$-complement} of $(X,\Delta)$ over $Z$, if the following conditions are satisfied:
\begin{enumerate}
    \item $(X,B)$ has log canonical singularities, and 
    \item $N(K_X+B)\sim_Z 0$ holds over a neighborhood of $z\in Z$.
\end{enumerate}
In the case that $X\rightarrow Z$ is the identity morphism, then a $N$-complement is a log canonical pair $(X,B)$ for which $B\geq \Delta$ and $N(K_X+B)\sim 0$ around $x\in X$.
}
\end{definition}

\subsection{Rationally connected varieties}
In this subsection, we prove a lemma regarding rationally connected varieties admitting a log Calabi-Yau klt pair structure. 
In this article, we only use a weak form of Lemma~\ref{lem:RC-curve} where (3) is not considered.
We expect the full strength of the lemma to be useful in forthcoming research.

\begin{lemma}\label{lem:acyclic}
Let $X$ be a projective rationally connected variety with klt singularities.
Then, $\mathcal{O}_X$ is an acyclic sheaf.
\end{lemma}

\begin{proof}
The statement is known for smooth projective rationally connected varieties.
Let $Y\rightarrow X$ be a resolution of $X$. Note that $Y$ is also rationally connected~\cite[Corollary 4.18]{Deb01}.
Since $X$ has klt singularities, by the Leray spectral sequence, we have that $H^i(X,\mathcal{O}_X)\simeq H^i(Y,\mathcal{O}_Y)=0$ for every $i\geq 1$. 
\end{proof}

\begin{lemma}\label{lem:RC-curve}
Let $n$ and $N$ be two positive integers.
There exists a constant $k:=k(n,N)$,
only depending on $n$ and $N$,
satisfying the following.
Let $X$ be a $n$-dimensional rationally connected variety and $B$ a boundary on $X$ for which $(X,B)$ is a log Calabi-Yau klt pair with $N(K_X+B)\sim 0$.
Le $E$ be a $\qq$-Cartier $\qq$-divisor which is not numerically trivial.
Then, there exists a curve $C\subset X$ satisfying the following conditions:
\begin{enumerate}
    \item $C$ is contained in the smooth locus of $X$, 
    \item $C$ is either disjoint from $B$ or it intersects $B$ transversally in at most $k$ points, and
    \item the divisor $E$ intersects $C$ non-trivially.
\end{enumerate}
\end{lemma}

\begin{proof}
By ~\cite[Theorem 1.4]
{BDCS20}, we know that 
$(X,B)$ is log bounded up to flops.
Let $(\mathcal{X},\mathcal{B}) \rightarrow T$ be a log bounding pair, up to flops, for the projective varieties as in the statement of the lemma.
By Noetherian induction, we may assume that $\mathcal{X}\rightarrow T$ is a smooth projective morphism.
Moreover, we may assume that every fiber $(\mathcal{X}_t,\mathcal{B}_t)$ is a klt log Calabi-Yau pair with
$N(K_{\mathcal{X}_t}+\mathcal{B}_t)\sim 0$ and that $\mathcal{X}_t$ is rationally connected.
Let $\mathcal{X}'\rightarrow T$ be a small $\qq$-factorialization of $\mathcal{X}\rightarrow T$.
Up to shrinking $T$, we may assume that $\mathcal{X}'_t\rightarrow \mathcal{X}_t$ is a small $\qq$-factorialization for every $t$.
Hence, we may assume that fibers are $\qq$-factorial.
By~\cite[Proposition 2.8.(1)]{HX15}, 
up to a finite base change,
we may assume that the homomorphism
induced by restriction
\begin{equation}\label{n1-isom}
N_1(\mathcal{X}'/T) \rightarrow
N_1(\mathcal{X}'_t)
\end{equation} 
is an isomorphism for every $t\in T$.
By duality, we obtain that 
\begin{equation}\label{n^1-isom}
N^1(\mathcal{X}'/T) \rightarrow
N^1(\mathcal{X}'_t)
\end{equation}
is an isomorphism for every $t\in T$.
We note that~\cite[Proposition 2.8.(1)]{HX15} requires the morphism to be of Fano type. However, the acyclicity of the structure sheaf of the fibers suffices for the proof.
This holds by Lemma~\ref{lem:acyclic}.

Let $d$ be the relative dimension of $\mathcal{X}'\rightarrow T$
and let $r$ be the relative Picard rank of $\mathcal{X}$ over $T$.
For each $t$, the map
$\phi_t \colon N^1(\mathcal{X}'_t)\rightarrow N_1(\mathcal{X}'_t)$ sending an ample divisor $A$ to $A^{d-1}$ is a 
local diffeomorphism at every ample class (see, e.g.,~\cite[Remark 2.3]{LX16}).
By isomorphisms~\eqref{n1-isom} and~\eqref{n^1-isom}, 
we have that 
\[
\Phi_t \colon 
N^1(\mathcal{X}'/T) \rightarrow 
N^1(\mathcal{X}'_t) \rightarrow 
N_1(\mathcal{X}'_t) \rightarrow
N_1(\mathcal{X}'/T)
\]
is a local diffeomorphism around every relatively ample class.

Fix $t_0\in T$ and let $H$ be a relatively ample class on $\mathcal{X}'$ over $T$.
Let $A_1,\dots, A_r$ be relatively ample classes on $N^1(\mathcal{X}'/T)$ close to $H$ so that $\Phi_{t_0}(A_1),\dots,
\Phi_{t_0}(A_r)$ are linearly independent.
Assume that for some $t\in T$, the curve classes \[
A_1^{d-1}|_{\mathcal{X}'_t},
\dots
A_r^{d-1}|_{\mathcal{X}'_t}
\]
are not linearly independent.
Let $D_t$ be a non-zero class in $N^1(\mathcal{X}'_t)$ so that 
\[
D_t\cdot A_1^{d-1}|_{\mathcal{X}'_t} =\dots =
D_t \cdot A_r^{d-1}|_{\mathcal{X}'_t} =0.
\]
Let $D$ be the unique non-zero lifting of $D_t$ to $N^1(\mathcal{X}'/T)$.
The intersection numbers
\[
0=D_t\cdot A_i^{d-1}|_{\mathcal{X}'_t} =
(D\cdot A^{d-1})|_{\mathcal{X}'_t}
\]
are independent of the fiber.
We conclude that $D|_{X_{t_0}}$ is numerically trivial. Then 
$D|_{X_t}$ must be numerically trivial as well.
This is a contradiction.
We conclude that $\Phi_t(A_1),\dots,\Phi_t(A_r)$ span $N_1(\mathcal{X}'_t)$ for every $t$.
Observe that for every closed subvariety $Z\subset \mathcal{X}'_t$ of codimension at least $2$, we can choose the irreducible curves $\Phi_t(A_1),\dots,\Phi_t(A_r)$ to be disjoint from $Z$.
Hence, the strict transforms of these curves on $\mathcal{X}_t$ generate $N_1(\mathcal{X}_t)$.
Furthermore, the intersection numbers 
\[
\mathcal{B}_t \cdot \Phi_t(A_1), 
\dots ,
\mathcal{B}_t \cdot \Phi_t(A_r)
\]
are bounded above by a constant only depending on the log bounding family.
Since the log bounding family only depends on $n$ and $N$,
we conclude that there exists a constant $k:=k(n,N)$ bounding the above intersection numbers.
By construction, we know that each $\Phi_t(A_i)$ intersects $\mathcal{B}_t$ transversally.
By abuse of notation, we denote by
$\Phi_t(A_i)$ the corresponding curves on $\mathcal{X}_t$.

Let $(X,B)$ be a log pair as in the statement of the proposition.
We know that there exists a composition of flops
$\phi\colon (X,B)\dashrightarrow (\mathcal{X}_t,\mathcal{B}_t)$ so 
that $(\mathcal{X}_t,\mathcal{B}_t)$ is isomorphic to a fiber of the log bounding family
$(\mathcal{X},\mathcal{B})\rightarrow T$.
By the first paragraph, we know that there exists curves 
$\Phi_t(A_1),\dots, \Phi_t(A_r)$ which are movable on $\mathcal{X}_t$ and generate $N_1(\mathcal{X}_t)$.
Furthermore, by replacing each $\Phi_t(A_i)$ with a numerically equivalent curve, we may assume that the following conditions are satisfied:
\begin{enumerate} 
    \item[(i)] Each $\Phi_t(A_i)$ lies in the smooth locus of $\mathcal{X}_t$, 
    \item[(ii)] each $\Phi_t(A_i)$ is disjoint from ${\rm Ex}(\phi^{-1})$, and
    \item[(iii)] $\mathcal{B}_t \cdot \Phi_t(A_i)\leq k$ for each $i$.
\end{enumerate}
Since $E$ is not numerically trivial in $X$,
then its push-forward $E_t$ on $\mathcal{X}_t$ is not numerically trivial.
Thus, $E_t\cdot \Phi_t(A_i)\neq 0$ for some $i\in \{1,\dots,r\}$.
We may assume that $i=1$.
Let $C$ be the strict transform of 
$\Phi_t(A_1)$ in $X$.
By (i) and (ii), we have that 
$C$ lies in the smooth locus of $X$.
This gives us (1).
On the other hand, we have that 
\[
E\cdot C = E_t \cdot \Phi_t(A_1) \neq 0.
\]
Furthermore, 
we have that
\[
B\cdot C = \mathcal{B}_t\cdot \Phi_t(A_1) \leq k,
\]
and the intersection of $B$ with $C$ is transversal. This gives us $(2)$ and $(3)$, concluding the proof.
\end{proof}

\subsection{Complexity}
In this subsection, we recall the definition of the complexity of a log pair. 
We recall the main theorem about the complexity 
which states that the complexity is always non-negative (see, e.g.,~\cite{Kol92,BMSZ18,RS21}).

\begin{definition}
{\em
Let $(X,B;x)$ be a log canonical pair.
Write $B=\sum_{i=1}^r b_iB_i$, where the $b_i$'s are non-negative real numbers and the 
$B_i$'s are effective Weil divisors.
The {\em complexity} of $(X,B)$ at $x$ is defined to be 
\[
c(X,\Delta;x):= 
\dim X + \rho(X_x) - \sum_{i=1}^r b_i.
\] 
Here, $X_x$ is the localization of $X$ at the closed point $x$.
}
\end{definition}

The following theorem is proved in~\cite[Theorem 1]{RS21}.

\begin{theorem}
\label{thm:comp}
Let $(X,B;x)$ be a log canonical singularity.
Then $c(X,B;x)\geq 0$.
Furthermore, if $c(X,B;x)=0$, then
$(X,\lfloor B\rfloor)$ is formally toric at the point $x$.
In particular, if $B$ is a reduced divisor, then it has at most $\dim X + \rho(X_x)$ prime components.
\end{theorem} 

\subsection{Dual complexes} 
In this subsection, we recall the definition of the dual complex of a log canonical pair.
We refer the reader to~\cite{KX16,dFKX17} for the classics about dual complexes of log Calabi-Yau pairs and singularities.
In~\cite{FS20,Bir20}, the authors prove connectedness results about the dual complexes of log Calabi-Yau pairs.

\begin{definition}
\label{def:dual-complex}
{\em 
Let $E$ be a projective scheme.
Assume that $E$ is pure-dimensional
and let $E_1,\dots,E_r$ be its irreducible components.
We assume that each $E_i$ is a normal variety.
Moreover, for every $J\subset \{1,\dots,r\}$,
if the intersection $\cap_{j\in J}E_j$ is non-empty, then every connected component of this set is irreducible and has codimension equal to $|J|-1$ as a subvariety of $E$.
The {\em dual complex}
$\mathcal{D}(E)$ of $E$ is defined as follows.
The vertices $v_1,\dots,v_r$ are in one-to-one correspondence with the prime components $E_1,\dots,E_r$.
Given an irreducible component $W$ of the intersection $\cap_{j\in J}E_j$, we associate a cell $v_W$ of dimension $|J|-1$.
Note that for each $i\in J$, the variety $W$
is contained in a unique irreducible component of $\cap_{j\in J\setminus \{i\}}E_i$.
This determines the gluing of the cell $v_W$.
The dual complex $\mathcal{D}(E)$ is a CW complex.

Let $(X,B)$ be a log canonical pair.
Let $(Y,B_Y)$ be a dlt modification of $(X,B)$ (see Lemma~\ref{lem:existence-dlt-mod}).
We define the dual complex of $(X,B)$ to be 
\[
\mathcal{D}(X,B):=
\mathcal{D}(\lfloor B_Y\rfloor).
\]
Here, 
the dual complex of $\lfloor B_Y\rfloor$ is defined as in the first paragraph.
By~\cite[Lemma 2.32]{FS20}, we know that $\mathcal{D}(X,B)$ is independent of the dlt modification up to simple homotopy equivalence.
}
\end{definition}

The following lemma will be used to control the dual complex of lc singularities of regularity one (cf.~\cite[Corollary 24]{dFKX17}).

\begin{lemma}
\label{lem:dual-comp-coll}
Let $(X,\Delta;x)$ be a klt singularity.
Let $Y\rightarrow X$ be a projective birational morphism which is an isomorphism in $X\setminus \{x\}$.
Let $E$ be the reduced exceptional divisor.
Assume that $(Y,E+\Delta_Y)$ has dlt singularities, where
$\Delta_Y$ is the strict transform of $\Delta$ on $Y$.
Then, the dual complex $\mathcal{D}(Y,E)$ is collapsible.
\end{lemma}

\subsection{Regularity} 
In this subsection, we recall the definition of the regularity of a klt singularity.
We prove several lemmas about the regularity of klt singularities and their relation with the theory of complements.

\begin{definition}
{\em 
Let $(X,B)$ be a log canonical pair.
In~\cite[Lemma 4.2]{Mor20}, we showed that we can find a simplicial representative of $\mathcal{D}(X,B)$.
Furthermore, we can assume that all maximal dimensional simplices of
$\mathcal{D}(X,B)$ have 
the same dimension $r$.
The non-negative interger $r$ is 
defined to be the {\em regularity} of $\mathcal{D}(X,B)$.
We denote the regularity by $\mathcal{D}(X,B)$ (see~\cite[Definition 7.11]{Sho00}).
In the case that $\mathcal{D}(X,B)$ is empty, 
for instance, if $(X,B)$ is klt, then we set
${\rm reg}(X,B):=-\infty$.
From the definition, it follows that 
${\rm reg}(X,B)\in \{-\infty,0,\dots,\dim X-1\}$.
The upper bound on the regularity follows from Theorem~\ref{thm:comp}.

Let $(X,\Delta;x)$ be a klt singularity.
We define its regularity
as the maximum among the regularities
${\rm reg}(X,B;x)$ so that
$(X,B)$ is log canonical at $x$ and $B\geq \Delta$.
This means that 
\[
{\rm reg}(X,\Delta;x) := \max
\left\{ 
{\rm reg}(X,B;x) \mid 
(X,B;x) \text{ is lc and $B\geq \Delta$ }
\right\}. 
\]
In other words, the regularity of a pair 
is the maximum of the regularities 
among its $\qq$-complements.
Note that the in 
the case of singularities, 
we have that 
${\rm reg}(X,\Delta;x)\in \{0,\dots,\dim X-1\}$.
Indeed, we can always produce a log canonical center through the point $x$.
}
\end{definition}

The following lemma shows that a divisor computing the regularity must also compute a log canonical threshold at the closed point.

\begin{lemma}\label{lem:reg-comp-lcc}
Let $(X,\Delta;x)$ be a klt singularity.
Assume that $X$ is $\qq$-factorial.
Let $B\geq \Delta$ be an effective divisor so that 
\[
{\rm reg}(X,\Delta;x)=
{\rm reg}(X,B;x).
\]
Then, $x$ is a log canonical center of $(X,B;x)$.
\end{lemma}

\begin{proof}
Let $r={\rm reg}(X,\Delta;x)$.
Without loss of generality, we may assume that all the log canonical centers of $(X,B;x)$ pass through $x$.
Let $(Y,B_Y)$ be a dlt modification of $(X,B)$.
We denote by $p\colon Y\rightarrow X$ the corresponding projective birational morphism.
Then, we can find $r+1$ prime components $E_1,\dots,E_{r+1}$ of $\lfloor B_Y\rfloor$
with non-empty intersection.
Let $Z$ be the image of $E_1\cap\dots\cap E_{r+1}$ on $X$.
We can find an effective $\qq$-Cartier divisor $\Gamma$ through $x\in X$ so that $(X,B+\Gamma)$ is log canonical and $x$ is a log canonical center of $(X,B+\Gamma)$.
Write
\[
K_Y+B_Y+\Gamma_Y = p^*(K_X+B+\Gamma).
\] 
Then, there is a log canonical center of $(Y,B_Y+\Gamma_Y)$ which maps onto $x$.
By~\cite[Theorem 1]{Mor20}, we can find a projective birational morphism
$\phi\colon Y'\rightarrow Y$ which extracts a unique log canonical place of $(Y,B_Y+\Gamma_Y)$ whose center on $X$ is $x$.
Let $E_0$ be such prime divisor.
We run a $(-E_0)$-MMP over the base.
Let $Y'\dashrightarrow Y''$ be this minimal model program.
Let $(Y'',B_{Y''}+\Gamma_{Y''})$ be the log pull-back of $(X,B+\Gamma)$ to $Y''$.
We denote by $E''_0,\dots,E''_{r+1}$ the strict transforms of the components $E_0,E_1,\dots,E_{r+1}$, respectively.
Then, the fiber of $Y''\rightarrow X$ over $x$ equals $E''_0$, which is a log canonical center of $(X,B+\Gamma)$.
By construction, this minimal model program is an isomorphism over the strict transform of $Z$.
Hence, $E''_1\cap \dots \cap E''_{r+1}$ is non-trivial and must intersect the fiber over $x$.
Thus, we have that
\[
E''_0\cap \dots\cap E''_{r+1}\neq \emptyset.
\]
We conclude that ${\rm reg}(X,B+\Gamma;x) > {\rm reg}(X,B;x)$, leading to a contradiction.
\end{proof}

The following lemma asserts that given a klt singularity, we can find a $N$-complement which computes its regularity, where $N$ only depends on the dimension of the singularity (and possibly the coefficients set).

\begin{lemma}\label{lem:reg-bounded-comp}
Let $n$ be a positive integer.
Let $\Lambda\subset \qq$ be a set satisfying the descending chain condition with rational accumulation points.
There exists a constant $N:=N(n,\Lambda)$, only depending on $n$ and $\Lambda$, satisfying the following.
Let $(X,\Delta;x)$ be a $n$-dimensional klt singularity
so that ${\rm coeff}(\Delta)\subset \Lambda$.
Then, there exists a $N$-complement of $(X,\Delta;x)$ that computes
${\rm reg}(X,\Delta;x)$.
\end{lemma}

\begin{proof}
Let $r={\rm reg}(X,\Delta;x)$.
Let $B'\geq \Delta$ be a boundary that computes the regularity of $(X,\Delta;x)$.
Let $(Y,B'_Y)$ be a dlt modification of $(X,B';x)$.
We denote by $p\colon Y\rightarrow X$ the corresponding projective birational morphism.
By construction, we can find $r+1$ prime components $E_1,\dots,E_{r+1}$ of $\lfloor B'_Y\rfloor$ that intersect non-trivially.
We define 
\[
\Delta_{Y,d}:= p^{-1}_*\Delta+ \left(1-\frac{1}{d}\right){\rm Ex}(p),
\] 
for every $d\in \zz_{\geq 1}$.
Then, the pair $(Y,\Delta_{Y,d})$ is a klt pair.
By~\cite[Theorem 2.19]{Mor20a}, we can find a $N$-complement for 
$(Y,\Delta_{Y,d})$ over $x$, where $N$ only depends on $n$ and $\Lambda$.
We let $(Y,\Gamma_{Y,d})$ be the $N$-complement.
Then, we have that $(Y,\Gamma_{Y,d})$ is log canonical and 
\[
N(K_Y+\Gamma_{Y,d})\sim_X 0.
\]
Let $(X,\Gamma_d)$ be the log pair obtained by pushing-forward $(Y,\Gamma_{Y,d})$ to $X$.
Observe that $(X,\Gamma_d)$ is log canonical and 
$N(K_X+\Gamma_d)\sim 0$ around the point $x\in X$.
If $d>N$, then for each $i\in \{1,\dots,r+1\}$, we have that 
\[
1-{\rm coeff}_{E_i}(\Gamma_{Y,d})=a_{E_i}(X,\Gamma_d)=0.
\] 
We conclude that $(X,\Gamma_d;x)$ is a $N$-complement of $(X,\Delta;x)$ of regularity $r$ provided that $d>N$.
\end{proof}

\subsection{Regularity one} In this subsection, we prove some lemmas regarding the geometry of klt singularities and log Calabi-Yau pairs of regularity one.
The following Lemma characterizes the combinatorial structure of the possible log canonical centers of an lc singularity of regularity one.

\begin{lemma}\label{lem:class-dual-comp}
Let $(X,B;x)$ be a log canonical pair of regularity one.
Assume that $(X,\Delta;x)$ is klt for some $\Delta \leq B$.
Let $(Y,B_Y)$ be the dlt modification of $(X,B;x)$ (see Lemma~\ref{lem:existence-dlt-mod}).
Assume that $(Y,B_Y)$ has at least one divisorial log canonical center mapping onto $\{x\}$
and that every log canonical center of $(X,B;x)$ passes through $x$.
Write $E_1,\dots,E_r$ for the prime components of $\lfloor B_Y\rfloor$.
Then, one of the following holds:
\begin{enumerate}
    \item [(i)] $\mathcal{D}(Y,B_Y)$ is a closed interval: $E_1$ maps onto a log canonical center $Z_1\supsetneq x$, $E_r$ maps onto a log canonical center
    $Z_r\supsetneq x$, 
    $Z_1\neq Z_r$, and each $E_i$ with $i\in \{2,\dots,r-1\}$ maps onto $x$,
    \item[(ii)] $\mathcal{D}(Y,B_Y)$ is a circle: $E_1$ maps onto a log canonical center $Z\supsetneq x$ and every other $E_i$ maps onto $x$, 
    \item [(iii)] $\mathcal{D}(Y,B_Y)$ is a closed interval:
    $E_1$ maps onto a log canonical center
    $Z_1\supsetneq x$ and each $E_i$ 
    with $i\geq 2$ maps onto $x$, or 
    \item [(iv)] $\mathcal{D}(Y,B_Y)$ is a closed interval: each $E_i$ with $i\in \{1,\dots,r\}$ maps onto $x$.
\end{enumerate}
\end{lemma}

\begin{proof}
Let $\phi\colon Y\rightarrow X$ be the projective birational morphism giving the dlt modification.
Since $(X,B;x)$ has regularity one, then each divisorial log canonical center of $(Y,B_Y)$ intersects at most two other divisorial log canonical centers.
Hence, $\mathcal{D}(Y,B_Y)$ is either a circle or a closed interval.

Let $E_k$ be a component of $\lfloor B_Y\rfloor$ which maps onto $x$.
We run a $(-E_k)$-MMP over $X$ which terminates with a good minimal model for $-E_k$.
Note that this minimal model program is an isomorphism over 
$X\setminus \{x\}$.
Furthermore, since every curve contracted or flipped by this MMP is $E_k$-positive, then $E_k$ is not contracted. 
Let $Y^*\rightarrow X$ be the model where this minimal model program terminates.
We denote by $E_k^*$ the strict transform of $E_k$ on $Y^*$.
We let $\phi^*\colon Y^*\rightarrow X$ be the induced projective birational morphism.
We show that ${\phi^*}^{-1}(x)=E_k^*$.
Indeed, assume that there is an irreducible component of ${\phi^*}^{-1}(x)$ which is not $E_k^*$.
Since ${\phi^*}^{-1}(x)$ is connected, we may assume that such component intersects $E_k^*$ non-trivially.
Thus, we can find a curve $C$ in such component, which intersects $E_k^*$, and is not contained in $E_k^*$.
Then, we have that $-E_k^*\cdot C<0$, contradicting the fact that
$-E_k^*$ is semiample
over the base $X$.

We denote by
$(Y^*,B_{Y^*})$ the log pull-back of $(X,B)$ to $Y^*$.
For each $E_i$ in $Y$, we denote by
$E_i^*$ its push-forward on $Y^*$.
If $E_i$ is contracted, then we let $E_i^*$ be the trivial divisor.
We claim that $\lfloor B_{Y^*}-E^*_k \rfloor$ has at most two prime components.
Note that $(Y^*,B_{Y^*})$ has regularity one. Then, by Lemma~\ref{lem:from-reg-1-to-0}, the log pair
\[
K_{E^*_k} + B_{E_k^*}\sim_\qq 
K_{Y^*}+B_{Y^*}|_{E_k^*}
\]
has regularity zero.
In particular, the divisor
$\lfloor B_{E_k^*}\rfloor$
either has two disjoint prime components or a unique prime component.
Let $E_i^*$ be a prime divisor on $Y^*$ so that its image on $X$ properly contains the point $x$.
Then, $E_i^*$ intersects $E_k^*$ non-trivially.
Furthermore, we have that
\begin{equation}
\label{eq:int-E_k} 
\supp\left(E_i^*\cap E_k^*\right)
\subseteq \lfloor B_{E_k^*}\rfloor.
\end{equation} 
On the other hand, if $E_j^*$ and $E_i^*$ are two different prime divisors, then the intersection
$E_i^*\cap E_j^*\cap E_k^*$ has codimension at least three.
Otherwise, we would obtain a log canonical surface singularity with negative complexity, leading to a contradiction (see Theorem~\ref{thm:comp}).
From the containment~\eqref{eq:int-E_k}, we conclude that there are at most two components $E_i^*$ and $E_j^*$ which are not contracted in the minimal model program $Y\dashrightarrow Y^*$.

If the exceptional locus of
$\phi^*$ only consists of the divisor $E^*_k$, then all the divisors $E_i$ are contracted by the minimal model program
$Y\dashrightarrow Y^*$.
This means that each $E_i$ maps onto $x$.
By Lemma~\ref{lem:dual-comp-coll}, the dual complex $\mathcal{D}(Y,B_Y)$ is collapsible.
Then, it must be a closed interval.
Thus, we are in the situation of 
(iv).
If there is exactly one prime divisor $E_i^*$ in the exceptional locus of $\phi^*$ which is different from $E_k^*$,
then there are two possibilities:
\begin{enumerate}
    \item [(a)] The intersection $E_i^* \cap E_k^*$ is connected.
In this case, the intersection must be irreducible.
Furthermore, $E_i$ must intersect a unique prime component of $\lfloor B_Y -E_i \rfloor$. 
Thus, $E_i$ corresponds to an  end-point of the closed interval $\mathcal{D}(Y,B_Y)$.
Hence, we are in the situation of (iii).
    \item [(b)] The intersection $E_i^* \cap E_k^*$ is disconnected.
    In this case, the intersection must be two irreducible divisors in $E_k^*$.
    Hence, we are in the situation of (ii).
\end{enumerate}
Finally, 
assume there are exactly two prime divisors $E_i^*$ and $E_j^*$ in the exceptional locus of $\phi^*$ which are different from $E_k^*$.
Then, each intersection
$E_k^*\cap E_i^*$ and
$E_k^*\cap E_j^*$ is irreducible,
these intersections are disjoint, and 
$E_i\cap E_j=\emptyset$.
Furthermore, the divisor
$E_i$ (resp. $E_j$) must intersect a unique prime component of 
$\lfloor B_Y-E_i\rfloor$
(resp. $\lfloor B_Y-E_j\rfloor$).
We conclude that the divisors $E_i$ and $E_j$ correspond to the end-points of the closed interval $\mathcal{D}(Y,B_Y)$.
Thus, we are in the situation of (i).
\end{proof}

The following lemma will allow us to control the log discrepancy at certain potentially log canonical places of a klt singularity
of regularity one.

\begin{lemma}\label{lem:not-point}
Let $n$ and $N$ be two positive integers.
Let $\Lambda$ be a set satisfying the descending chain condition.
There exists a set $\mathcal{M}(n,N,\Lambda)$
satisfying the ascending chain condition, 
which only depends on $n,N$ and $\Lambda$, satisfying the following.
Let $(X,B;x)$ be a log canonical pair of regularity one.
Assume that $\{x\}$ is a log canonical center of $(X,B)$.
Assume that $N(K_X+B)\sim 0$.
Let $E$ be a log canonical place of $(X,B;x)$ which maps to $Z\supsetneq x$.
Let $\Delta$ be an effective divisor with $\Delta\leq B$ so that $(X,\Delta;x)$ is klt and ${\rm coeff}(\Delta)\subset \Lambda$.
Then, we have that 
\[
a_E(X,B)\in \mathcal{M}(n,N,\Lambda).
\] 
Furthermore, if the set $\Lambda$ is finite, then the set $\mathcal{M}(n,N,\Lambda)$ only accumulates to zero.
\end{lemma}

\begin{proof}
Let $(Y,B_Y)$ be a $\qq$-factorial dlt modification of $(X,B;x)$ (see Lemma~\ref{lem:existence-dlt-mod}).
Let $\phi\colon Y\rightarrow X$ be the
corresponding projective birational morphism.
Since $x$ is a log canonical center, we may assume there exists a prime 
component of $\lfloor B_Y\rfloor$ which maps onto $x$.
We call such component $E_0$.
We claim that $E$ is the unique log canonical place of $(X,B;x)$ whose image on $X$ contains $Z$.
Assume it is not.
Let $E'$ be a log canonical place of $(X,B;x)$, different from $E$, which maps to $Z$.
By the connectedness theorem~\cite[Theorem 1.2]{Bir21}, we can assume that $E'$ intersects $E$ over $Z$.
In particular, we may assume that
$E\cap E'$ dominates $Z$.
Since $Z$ contains $x$, then
$E\cap E'$ intersects the fiber over $x$.
We can run a $(-E_0)$-MMP over the base.
This minimal model program is an isomorphism at the generic point of $E\cap E'$.
After finitely many steps, we obtain a model in which
$E\cap E'\cap E_0$ is non-empty.
This implies that the regularity of $(X,B;x)$ is at least two, leading to a contradiction. 
Hence, we may assume that $E$ is the only log canonical place which maps onto $Z$.

From now on, we can localize at a general point of $Z$, so we may assume that $E$ is the only log canonical center.
Let $\phi_E\colon E\rightarrow Z$ be the induced fibration.
By~\cite[Theorem 1.2]{HM07}, we conclude that the general fiber of $E\rightarrow Z$ is rationally chain connected.
By the previous paragraph, the general fiber of $E\rightarrow Z$ is klt.
Hence, it is rationally connected by~\cite[Corollary 1.8]{HM07}.
Let $(E,B_E)$ be the log pair induced by adjunction of $(Y,B_Y)$ to $E$.
We have that $(E,B_E)$ is klt, $E$ is rationally connected, and 
\[
N(K_E+B_E)\sim 0.
\]
By~\cite[Theorem 1.4]{BDCS20}, we conclude that $(E,B_E)$ is log bounded up to flops.
By Lemma~\ref{lem:RC-curve}, we can find a curve $C\subset E$ which satisfies the following conditions:
\begin{enumerate}
    \item $C$ lies in the smooth locus of $E$, and 
    \item $C$ is either disjoint from $B_E$ or it intersects $B_E$ transversally at most in $k$ points.
\end{enumerate}
Here, $k$ is a constant, only depending on $n$ and $N$.
Write $B_Y=E+B'_Y$.
Note that the coefficients of $B'_Y$ belongs to the finite set 
$\zz\left[\frac{1}{N}\right]\cap [0,1]$.
Let $(Y,\Delta_Y)$ be the log pull-back of $(X,\Delta)$ to $Y$.
Write $\Delta_Y = \alpha E +\Delta_Y'$.
Note that the coefficients of
 $\Delta'_Y$ belongs to the set $\Lambda$ which satisfies the descending chain condition.
We conclude that 
\[
\alpha = \frac{ C\cdot (B'_Y -\Delta'_Y)}{m}
\] 
belongs to a set satisfying the ascending chain condition.
Furthermore, if $\Lambda$ is a finite set, then 
$B'_Y-\Delta'_Y$ has coefficients in a finite set.
Hence, the possible values of
$\alpha$ can only accumulate to zero.
\end{proof}

\begin{lemma}\label{lem:from-reg-1-to-0}
Let $(X,B)$ be an lc pair of regularity one.
Let $E$ be a prime component of $\lfloor B\rfloor$.
Let $(E,B_E)$ be the pair obtained by adjunction of $(X,B)$ to $E$.
Then, $(E,B_E)$ has regularity zero.
In particular, $\lfloor B_E\rfloor$ has either one or two disjoint components. 
\end{lemma}

\begin{proof}
Assume that $(E,B_E)$ has regularity at least one.
Let $(Y,B_Y)$ be a dlt modification of $(X,B)$.
We denote by $E_Y$ be the strict transform of $E$ on $Y$.
Let $(E_Y,B_{E_Y})$ be the log pair obtained by adjunction of $(Y,B_Y)$ to $E_Y$.
Then, $(E_Y,B_{E_Y})$ is a dlt pair. Thus, the projective birational morphism $E_Y\rightarrow E$ is a dlt modification of $(E,B_E)$.
By further blowing-up, we may assume that there is a bijection between prime components of $\lfloor B_{E_Y}\rfloor$ and components of $\lfloor B_Y-E_Y\rfloor$ that intersect $E_Y$.
If ${\rm reg}(E,B_E)\geq 1$, then we can find at least two prime components $E_1,E_2$ of $\lfloor B_{E_Y}\rfloor$ for which $E_1\cap E_2\neq \emptyset$.
Let $E_{Y,1}$ and $E_{Y,2}$ be two components of $\lfloor B_Y\rfloor$ that restrict to $E_1$ and $E_2$, respectively.
Then, we have that
$E_{Y,1}\cap E_{Y,2}\cap E_Y\neq \emptyset$.
This leads to a contradiction.
The last part of the statement follows from~\cite[Theorem 1.2]{Bir21}.
\end{proof}

\begin{corollary}\label{cor:two-comp-toric}
Let $(X,B)$ be an lc pair of regularity one. Let $E_1,E_2\subset \lfloor B\rfloor$ be two prime components which are $\qq$-Cartier divisors.
Assume that $E_1\cap E_2\neq \emptyset$ is connected.
Then $E_1\cap E_2=Z$ is irreducible, no other log canonical center of $(X,B)$ intersects $Z$ and $(X,B)$ is toric at the generic point of $Z$.
\end{corollary}

\begin{proof}
The irreducibility of $Z$ follows from Lemma~\ref{lem:from-reg-1-to-0} and the fact that every irreducible component of an intersection of log canonical centers is again a log canonical center.
The toroidality of $(X,B)$ at the generic point of $Z$ follows from Theorem~\ref{thm:comp}.
\end{proof}

\subsection{\texorpdfstring{$\pp^1$-links}{p1-links}}
In this subsection, we define the concept of $\pp^1$-link structure of a log Calabi-Yau pair.
We also prove a lemma regarding projective pairs of regularity zero.

\begin{definition}\label{def:p1-links}
{\em 
Let $(X,B)$ be a log Calabi-Yau admitting a fibration $X\rightarrow Z$.
We say that $X\rightarrow Z$ is a {\em $\pp^1$-link}
for $(X,B)$ if the following
conditions are satisfied:
\begin{enumerate}
    \item The general fiber of $X\rightarrow Z$ is isomorphic to $\pp^1$.
    \item the pair $(X,B)$ is plt, 
    \item $(X,B)$ is log Calabi-Yau over $Z$, i.e., 
    $K_X+B\sim_{\qq,Z}0$, and
    \item $\lfloor B\rfloor$ has two components that dominate the base.
\end{enumerate}
In particular, the restriction of $(X,B)$ to a general fiber is isomorphic to the log pair
$(\pp^1,\{0\}+\{\infty\})$.
}
\end{definition}

\begin{lemma}\label{MMP-reg-zero}
Let $(X,B)$ be a dlt log Calabi-Yau pair of regularity zero.
Assume that $\lfloor B\rfloor$ is disconnected.
Let $E_1$ and $E_2$ be the components of $\lfloor B \rfloor$ and $B'=B-\lfloor B\rfloor$.
Let $R$ be an extremal $(K_X+B')$-negative ray of the nef cone.
Then, $R$ is generated by a rational curve $C$ satisfying:
\[
K_X\cdot C=-2, \quad 
E_1\cdot C=1, \quad 
E_2\cdot C=1, 
\text{ and } \quad
B'\cdot C=0.
\] 
\end{lemma}

\begin{proof}
We can find an ample divisor $A$ on $X$ so that the cone 
\[
\overline{NE}_1(X)_{K_X+B'\geq 0} +
\overline{NM}_1(X)
\]
intersects
$(K_X+B'+A)^{\perp}$ at $R$.
We run a $(K_X+B'+A)$-MMP with scaling of $A$.
Since each step of this minimal model program is $A$-positive, 
then it is also a $(K_X+B')$-MMP.
Given that $K_X+B'$ is not pseudo-effective, this minimal model program terminates with a Mori fiber space.
Let $X\dashrightarrow X'$ be the minimal model program and $X'\rightarrow Z$ be the Mori fiber space.
Let $B'$ be the push-forward of $B$ to $X'$.
Every curve contracted by this MMP and every flipping curve intersect $E_1+E_2$ positively.
We conclude that no divisor $E_1$ neither $E_2$ is contracted by this minimal model program.
Let $E'_1$ and $E'_2$ be the strict transform of $E_1$ and $E_2$ in $X'$, respectively.
Note that $E'_1+E'_2$ is ample over $Z$, then at least one of the $E'_i$'s must dominate the base.
Assume $E'_1$ dominates the base.
If $E'_2$ is vertical over $Z$, then $E'_1\cap E'_2\neq \emptyset$, leading to a contradiction.
We conclude that $E'_1$ and $E'_2$ dominate $Z$.
Assume that the general fiber of $X'\rightarrow Z$ has dimension at least two. We call the general fiber $F$.
Then $E'_1|_F$ and $E'_2|_F$ are ample divisors and then they must intersect non-trivially. leading to a contradiction.
We conclude that the general fiber of $X'\rightarrow Z$ is one-dimensional. Then it is isomorphic to $\pp^1$.
We conclude that $(X',B')\rightarrow Z$ is a $\pp^1$-link.
Then, it suffices to take $C$ to be the strict transform of a general fiber of $X'\rightarrow Z$.
The strict transform of $C$ in $X$ generates the ray $R$.
Since $X\dashrightarrow X'$ is an isomorphism on a neighborhood of such general fiber, 
we conclude that the intersections 
$K_X\cdot C=-2$, $E_1\cdot C=E_2\cdot C=1$, and
$B'\cdot C=0$ hold.
\end{proof}

\begin{definition}\label{def:gen-p1}
{\em 
Let $(X,B)$ be a dlt log Calabi-Yau pair of regularity zero.
Let $R$ be an extremal $(K_X+B')$-negative ray
of the nef cone. 
The curve $C$ constructed in Lemma~\ref{MMP-reg-zero} will be called a {\em general $\pp^1$} of the log pair of regularity zero.
In Theorem~\ref{introthm:uniqueness}, we will see that $R$ is unique, 
so $C$ is unique up to numerical equivalence.
}
\end{definition} 

\subsection{Surface log discrepancies}
In this subsection, we prove a couple of lemmas regarding surface toric minimal log discrepancies.
For the material regarding toric geometry, we refer the reader to~\cite{Ful93,CLS11}.
For toric minimal log discrepancies see~\cite{Amb06}.
The first two lemmas are well-known so
we will skip their proofs. 

\begin{lemma}\label{lem:min-lattice}
Let $\sigma$ be a cone in $\qq^2$ spanned by the lattice vectors $v_1$ and $v_2$.
Let $w_1,\dots,w_c \in \zz^2$ give a regular decomposition of the cone.
Let $L$ be a positive linear function on $\sigma$.
Then, the minimizer of $L$ in 
\[
(\sigma\setminus \mathbb{Q}_{\geq 0} v_1) \cap \zz^2
\]
is attained in the finite set
$\{ w_1,\dots,w_c,v_2\}$.
\end{lemma}

\begin{lemma}\label{lem:toric-surf-mld}
Let $\sigma\subset \qq^2$ be a full-dimensional rational polyhedral cone. Let $v_1$ and $v_2$ be the lattice generators of the extremal rays of $\sigma$. 
Let $X(\sigma)$ be the corresponding affine toric surface and $x_0$ the torus invariant point. Let $T_1$ and $T_2$ be the torus invariant divisors of $X(\sigma)$.
Let $b_1$ and $b_2$ be two real numbers.
Then, we have that
\[
{\rm mld}(X(\sigma),b_1T_1+b_2T_2;x_0)=
\min_{u\in {\rm relint}(\sigma)\cap \zz^2}L(u),
\]
where $L$ is the unique linear function with
$L(v_1)=1-b_1$ and $L(v_2)=1-b_2$.
\end{lemma}

\begin{lemma}\label{lem:k-th-surf-toric-mld}
Let $\Lambda \subset \rr_{\leq 1}$ be a set satisfying the descending chain condition.
Let $k$ be a positive integer.
Let $\mathcal{M}_{T,\Lambda,k}$ be the set of $k$-th minimal log discrepancies of surface toric sub-pairs
$(X,\Delta;x)$ with 
${\rm coeff}(\Delta)\subset \Lambda$.
Then, the set
$\mathcal{M}_{T,\Lambda,k}$ satisfies the ascending chain condition.
Furthermore, if $\Lambda$ is finite, then its accumulation points only accumulate to zero.
\end{lemma}

\begin{proof}\label{cor:up-to-k-th-surf-toric-mld}
First, we prove the statement for $k=1$.
Let $-M$ be a lower bound of $\Lambda$.
We can assume that $M>0$.
Let $T_1$ and $T_2$ be the reduced prime toric divisors of $\Delta$.
Let $b_1$ and $b_2$ be the coefficients of $\Delta$ at $T_1$ and $T_2$ respectively.
Then, by Lemma~\ref{lem:toric-surf-mld}, we have that
\begin{equation}\label{eq:mld-mult}
{\rm mld}(X,\Delta;x) =
(M+1) {\rm mld}\left( 
X, 
\left(1- \frac{1-b_1}{M+1}\right)T_1 
+
\left(1 -\frac{1-b_2}{M+1}\right)T_2;
x
\right).
\end{equation} 
Note that 
\begin{equation}\label{eq:new-coeff}
0\leq 1-\frac{1-b_i}{M+1}\leq 1,
\end{equation} 
for $i\in \{1,2\}$.
We conclude that the sub-pair on the right side of equality~\eqref{eq:mld-mult} is actually a surface toric pair.
Furthermore, since $b_i \in \Lambda$, we conclude that the coefficients~\eqref{eq:new-coeff} belong to a set satisfying the descending chain condition, which only depends on $\Lambda$.
Then, since $M$ is a fixed number, we conclude that the value on the right side of equality~\eqref{eq:mld-mult} belongs to a set satisfying the ascending chain condition (see, e.g.,~\cite[Theorem 1.1]{Amb06}).
Furthermore, if $\Lambda$ is finite, then its accumulation points only accumulate to zero.
This shows the statement for $k=1$.

Now, assume that the statement holds for $k-1$.
Then, the set 
\begin{equation}\label{eq:union-mld}
\mathcal{M}_{T,\Lambda,\leq k-1}:=
\bigcup_{i=1}^{k-1} \mathcal{M}_{T,\Lambda,i}
\end{equation} 
satisfies the ascending chain condition.
Indeed, it is a finite union of sets satisfying the ascending chain condition.
Furthermore, the set $\mathcal{M}_{T,\Lambda,\leq k-1}$
only depends on $\Lambda$.
Moreover, if $\Lambda$ is finite, then the accumulation points of the set~\eqref{eq:union-mld} only accumulate to zero.
Let $(X,\Delta;x)$ be a singularity as in the statement.
We set
\[
\mathcal{C}_{T,\Lambda,\leq k-1}:=
\{1-m\mid m\in \mathcal{M}_{T,\Lambda,\leq k-1}\}. 
\]
We can find a toric projective birational morphism $\phi\colon Y\rightarrow X$ which extracts up to the $(k-1)$-th minimal log discrepancy of $(X,\Delta;x)$.
Hence, we can write 
\[
\phi^*(K_X+\Delta)=
K_Y+\Delta_Y,
\] 
where $(Y,\Delta_Y)$ is a toric sub-pair and 
\[
{\rm coeff}(\Delta_Y) \subseteq 
\Lambda':=\Lambda \cup 
\mathcal{C}_{T,\Lambda,\leq k-1}.
\]
Observe that $\Lambda' \subset \rr_{\leq 1}$ is a set satisfying the descending chain condition,
which only depends on $\Lambda$ and $k$.
Then, we conclude that
\[ 
\mathcal{M}_{T,\Lambda,k}
\subseteq \mathcal{M}_{T,\Lambda',1}.
\] 
By the case $k=1$,
we conclude that
$\mathcal{M}_{T,\Lambda,k}$
satisfies the ascending chain condition.
Furthermore, if $\Lambda$ is finite, then the accumulation points
of $\mathcal{M}_{T,\Lambda,k}$ can only accumulate to zero.
\end{proof}

\section{The cone of nef curves of log Calabi-Yau pairs} 
\label{sec:cone-nef}

In this section, we prove a structure theorem for the nef cone 
of log Calabi-Yau pairs with dlt singularities.
Furthermore, we prove the existence of the uniqueness of $\pp^1$-links for log Calabi-Yau pairs with regularity zero.

\begin{proof}[Proof of Theorem~\ref{introthm:nef-cone}]
Let $(X,B)$ be a log Calabi-Yau dlt pair
with $N(K_X+B)\sim 0$.
Let $\phi\colon X'\rightarrow X$ be a small $\qq$-factorialization of $X$.
Note that the cone of nef curves is preserved by the small $\qq$-factorialization, 
so we may replace $X$ with $X'$ and assume that $X$ itself is $\qq$-factorial.

We show that the log pair $(X,B')$ has $1/N$-log canonical singularities.
Assume this is not the case.
Let $E$ be a prime divisor over $X$ for which
$a_E(X,B') \in (0,N^{-1})$.
Then, we have that $a_E(X,B)=0$.
Indeed, we know that $N(K_X+B)\sim 0$, so the log discrepancies of $(X,B)$ belong to the set $\zz\left[\frac{1}{N}\right]$.
On the other hand, the generic point of each log canonical center of $(X,B)$ lies on the smooth locus of $X$ and it is disjoint from $B'$. 
Thus, for each $E$ with $a_E(X,B') \in (0,N^{-1})$, 
we have that $a_E(X)=a_E(X,B')\geq 1$.
This leads to a contradiction.
Thus, $(X,B')$ has $1/N$-log canonical singularities.

Let $R$ be a $(K_X+B')$-negative extremal ray of the cone of nef divisors.
We can find an ample divisor $A$ so that $K_X+B'+A$ is pseudo-effective and
$(K_X+B'+A)^{\perp}$ intersects the cone 
\[
\overline{NE}_1(X)_{K_X+B'\geq 0} 
+ 
\overline{NM}_1(X)
\]
exactly at $R$.
We can run a $(K_X+B')$-MMP with scaling of $A$ which terminates in a Mori fiber space.
Let $\psi \colon X\dashrightarrow X'$ be the minimal model program
and let $\phi\colon X'\rightarrow Z$ be the Mori fiber space.
Let $B'_{X'}$ and $A_{X'}$ be the strict transform of 
$B'$ and $A$ in $X'$, respectively.
By construction, $K_X'+B'_{X'}+A_{X'}$ is $\qq$-trivial over $Z$.
Let $E_{X'}$ be the strict transform of $E$ on $X'$.
Observe that $K_{X'}+B_{X'}+E_{X'}$ is $\qq$-trivial over $Z$
and that $-(K_{X'}+B_{X'})$ is ample over $Z$.
We conclude that $E_{X'}$ is ample over $Z$.
Let $F$ be a general fiber of $\phi\colon X'\rightarrow Z$. 
Let $B_F$ and $E_F$ be the restriction of $B_{X'}$ and $E_{X'}$ to $F$, respectively.
Since $(X,B)$ is $\frac{1}{N}$-log canonical and
$X\dashrightarrow X'$ is a $(K_X+B)$-MMP,
then $(X',B_{X'})$ is $\frac{1}{N}$-log canonical as well.
In particular, we have that $(F,B_F)$ is $\frac{1}{N}$-log canonical 
and $-(K_F+B_F)$ is ample.
Let $f$ be the dimension of $F$.
By~\cite[Theorem 1.1]{Bir21}, we conclude that $F$ belongs to a bounded family of $f$-dimensional varieties.
Since the coefficients of $B_F$ are at least $\frac{1}{N}$, by~\cite[Theorem 3.3]{FM20}, we conclude that
$(F,B_F)$ is log bounded.
Furthermore, the log pair $(F,B_F+E_F)$ is log Calabi-Yau and log bounded as well.
We denote by $E_{i,F}$ the restriction of the prime component $E_{i,X'}$ to $F$.
By~\cite[Lemma 2.19]{Mor18b}, we can find a curve $C_F$ on $F$, so that the following conditions are satisfied:
\begin{enumerate}
    \item the curve $C_F$ lies in the smooth locus of $F$,
    \item the curve $C_F$ is either disjoint from $B_F$ or intersect it transversally with $B_F\cdot C_F\leq k$, and
    \item for each $i$, the curve $C_F$ is either disjoint from $E_{i,F}$ or intersect it transversally at most in $k$ points.
\end{enumerate}
We can choose $C_F$ so that it is disjoint from $F\cap {\rm Ex}(\psi^{-1})$.
Then, the curve $C$ which is the strict transform of $C_F$ in $X$ satisfies the conditions $(1)$-$(4)$.
It suffices to prove that $(5)$ holds.
To do so, it is enough to prove that $C_F$ satisfies the analogous property.
Since $E_{X'}$ is ample over $C$, then $C_F$ must intersect at least one of its components.
We turn to prove that $C_F$ intersects at most $n+1$ of its components.
Up to re-ordering the $E_{i,X'}$, we may assume that there exists $r_0\leq r$ for which the divisors $E_{1,X'},\dots,E_{r_0,X'}$ dominate $Z$ and 
$E_{r_0+1,X'},\dots, E_{r,X'}$ are vertical over $Z$.
Since $\rho(X'/Z)=1$, 
we have that 
the components 
$E_{i,F}$ are $\qq$-linearly proportional. This means that 
\[
\dim_{\mathbb{Q}}
\langle 
E_{1,F}, \dots,
E_{r_0,F} 
\rangle = 1.
\]
Hence, we can compute the complexity
of the pair $(F,B_F+E_F)$ and obtain
\[
0\leq c(F,B_F+E_F) \leq f + 1 - r_0.
\]
The first inequality follows from~\cite[Theorem 4.5]{RS21} while the second inequality follows from~\cite[Definition 3.15]{RS21}.
Thus, we have that $r_0\leq f+1$.
We conclude that there are at most $f+1$ prime components of $E_{X'}$ that dominate $Z$.
This implies that the curve $C_F$ intersects at most $n+1$ of the components of $E_{X'}$.
Hence, the curve $C$ intersects at most
$n+1$ of the components $E_1,\dots,E_r$ of $\lfloor B\rfloor$.

We conclude that every $(K_X+B')$-negative extremal nef ray is spanned by a curve $C_i$ satisfying the conditions $(1)$-$(5)$.
Then, the theorem follows from~\cite[Theorem 1.3]{Leh12}.
\end{proof}

\begin{lemma}\label{lem:semiample-big}
Assume that Theorem~\ref{introthm:uniqueness} holds in dimension $d$.
Let $(F,B_F)$ be a $d$-dimensional log Calabi-Yau dlt pair of regularity zero.
Let $E_{1,F}$ and $E_{2,F}$ be the prime components of $\lfloor B_F\rfloor$ and $B'_F=B_F-\lfloor B_F\rfloor$.
Assume that the following conditions hold:
\begin{enumerate}
    \item $-(K_{F}+B'_F)$ is semiample and big, and
    \item $-(K_{F}+B'_F) \sim_\qq 2E_{1,F} \sim_\qq 2E_{2,F}$.
\end{enumerate}
Then, $F\simeq \pp^1$, $B'_F=0$, and  $E_{1,F}+E_{2,F}=\{0\}+\{\infty\}$.
\end{lemma}

\begin{proof}
We may replace $F$ with a small $\qq$-factorialization.
From now on, we may assume that $F$ is itself $\qq$-factorial.
Let $F\rightarrow F'$ be the ample model of $-(K_{F}+B'_F)$.
Let $B_{F'}$ be the push-forward of $B_{F}$ to $F'$.
Let $B'_{F'}$ be the push-forward of $B'_F$ to $F'$.
Since $-(K_{F}+B'_F)\sim_\qq 2E_{1,F}\sim_\qq 2E_{2,F}$, then every curve contracted by $F\rightarrow F'$ must intersect $E_{1,F}$ and $E_{2,F}$ trivially.
In particular, the divisors $E_{1,F}$ and $E_{2,F}$ are not contracted by the birational morphism $F'\rightarrow F$.
Then, the pair $(F',B_{F'})$ is log Calabi-Yau dlt pair of regularity zero and $-(K_{F'}+B'_{F'})$ is ample.
By Theorem~\ref{introthm:uniqueness} in dimension $d$, we conclude that $F'\simeq \pp^1$, 
$B'_{F'}=0$, and $E_{1,F'}+E_{2,F'}=\{0\}+\{\infty\}$.
This implies that $F$ must be one-dimensional, so the proof follows.
\end{proof}

\begin{proof}[Proof of Theorem~\ref{introthm:uniqueness}]
We will proceed by induction on the dimension.
If the dimension is one,
then it is clear that $X\simeq \pp^1$, $B'=0$,
and $B=\{0\}+\{\infty\}$.

Now, assume that $X$ has dimension at least two.
First, we will analyze the statement when the Picard rank is at most $2$.
Assume that $X$ has Picard rank one.
Let $E_1$ and $E_2$ be the components of $\lfloor B\rfloor$.
Since $\dim X \geq 2$, we can find a curve $C\subseteq E_1$.
Since $\rho(X)=1$, the divisor $E_2$ is an ample divisor and hence
$E_2\cdot C>0$.
This means that $E_1\cap E_2\neq \emptyset$, leading to a contradiction.

Assume that $X$ has Picard rank two and dimension at least two.
If the cone 
\begin{equation}\label{eq:cone-1}
\overline{NE}_1(X)_{K_X+B'\geq 0} + \overline{NM}_1(X)
\end{equation} 
has two extremal $(K_X+B')$-negative curves, 
it means that $-(K_X+B')$ must be an ample divisor.
Let $R_1$ and $R_2$ be the extremal rays of the cone of nef curves.
By Lemma~\ref{lem:curves-generating}, we can find $C_1$ and $C_2$, 
generating $R_1$ and $R_2$, respectively, such that
\[
K_X\cdot C_i = -2, \quad 
E_1 \cdot C_i=1, \quad 
E_2 \cdot C_i=1, \text{ and } \quad 
B'\cdot C_i=0,
\]
for each $i\in \{1,2\}$.
Since $C_1$ and $C_2$ span different rays in $N_1(X)$, we have that
\[
B'=0 \text{ and }
E_1\equiv E_2 \equiv -\frac{1}{2}K_X.
\]
This means that $E_1$ and $E_2$ are ample divisors.
This leads to a contradiction.
We conclude that the cone~\eqref{eq:cone-1} has a unique
$(K_X+B')$-negative extremal ray.

Assume that $X$ has Picard rank larger or equal than three.
If the cone~\eqref{eq:cone-1} has at least two $(K_X+B')$-negative
extremal rays, then it has a $(K_X+B')$-negative extremal face of dimension two.
We denote by $F$ such an extremal face of dimension two.
Then, we can find an ample divisor $A$ on $X$ so that
$K_X+B'+A$ is pseudo-effective and 
\[
(K_X+B'+A)^{\perp} \cap
\left( 
\overline{NE}_1(X)_{K_X+B'\geq 0} + \overline{NM}_1(X)
\right) =  F.
\]
We run a $(K_X+B'+A)$-MMP with scaling of $A$.
By~\cite[Corollary 1.4.2]{BCHM10}, this minimal model program terminates $\phi\colon X\dashrightarrow X_0$ with a good minimal model $X_0$.
Let $B'_0$ and $A_0$ be the push-forward of $B'$ and $A$ to $X_0$, respectively.
Let $Z$ be the ample model of the semiample divisor
$K_{X_0}+B'_0+A_0$.
Let $d$ be the relative dimension of $X_0\rightarrow Z$.

We claim that $\rho(X_0/Z)=2$.
Let $r=\rho(X_0/Z)$.
Let $\phi_* \colon N^1(X)\rightarrow N^1(X')$ be the map induced by push-forward. 
Observe that $\phi_*$ is surjective and maps the cone of effective divisors into the cone of effective divisors.
Let $\phi^*\colon N_1(X')\rightarrow N_1(X)$ be the dual of $\phi_*$.
Then, the homomorphism $\phi^*$ maps the cone of nef divisors into the cone of nef divisors.
A nef curve $C$ in $X'$ is contracted by $X'\rightarrow Z$ if and only if
\[
(K_{X_0}+B_0'+A_0)\cdot C = 
\phi_*(K_X+B+A) \cdot C =0.
\] 
Let $F'$ be the intersection of the cone of nef curves of $X_0$ with $(K_{X_0}+B_0'+A_0)^{\perp}$.
By duality, we have that $\phi^*(F')\subseteq F$.
Since $\phi^*$ is injective, we conclude that $F'$ has dimension at most two.
Hence, we conclude that $r\leq 2$.
On the other hand,
let $f_0\in F_0$ be a general point in a general fiber of $X_0\rightarrow Z$.
Let $x$ be the preimage of $f$ in $X$.
We show that we can find $C_1'\equiv C_1$ and $C_2'\equiv C_2$ so that 
\[
{\rm Ex}(X\dashrightarrow X') \cap C_i =\emptyset 
\text{ and }
x\in C_i 
\]
for each $i$.
Let $p\colon Y\rightarrow X$ and $q\colon Y\rightarrow X_0$ be a log resolution of the minimal model program.
Then, we have that
\[
p^*(K_X+B'+A) = q^*(K_{X_0}+B'_0+A_0) + E,
\]
where $E$ is an effective divisor.
By Theorem~\ref{introthm:nef-cone}, we can replace $C_1$ and $C_2$ by curves satisfying the following: 
\begin{itemize}
    \item the curves are disjoint from the set on which $p$ is not an isomorphism, 
    \item the curves are not contained in the divisorial locus of ${\rm Ex}(X\dashrightarrow X')$, and 
    \item both curves contain $x$.
\end{itemize}
Let $C_{Y,1}$ and $C_{Y,2}$ be the strict transform of $C_1$ and $C_2$ on $Y$, respectively.
Since $C_{Y,1}$ and $C_{Y,2}$ intersect $p^*(K_X+B'+A)$ trivially and
$q^*(K_{X_0}+B_0'+A_0)$ non-negatively, we conclude that
$C_{Y,1}$ and $C_{Y,2}$ must intersect both
$q^*(K_{X_0}+B_0'+A_0)$ and $E$ trivially.
In particular, $C_1$ and $C_2$ must be disjoint from 
the set ${\rm Ex}(X\dashrightarrow X_0)$.
Let $C_{1,0}$ and $C_{2,0}$ be the push-forward of $C_1$ and $C_2$ to $X_0$, respectively.
Then, the curves $C_{1,0}$ and $C_{2,0}$ are linearly independent in $N_1(X_0)$ and contracted to a point in $Z$.
Indeed, since both curves pass through $f_0$ and intersect 
$K_{X_0}+B'_0+A_0$ trivially, then they are contained in the fiber $F_0$.
Hence, we have that $r\geq 2$. 
Thus, we deduce that $r=2$ as claimed.

Let $E_{1,0}$ and $E_{2,0}$ be the push-forward of $E_1$ and $E_2$ to $X_0$, respectively.
Then, we conclude that 
\[
E_{1,0}+E_{2,0} \sim_{\qq,Z} -(K_{X_0}+B'_0) \sim_{\qq,Z} A_0
\]
is big over $Z$.
Since $E_{1,0}+E_{2,0}$ is big over $Z$, then either $E_{1,0}$ or $E_{2,0}$ must dominate $Z$.
Without loss of generality, we assume that $E_{1,0}$ dominates $Z$.
Let $Z_2$ be the image of $E_{2,0}$ in $Z$.
By the connectedness theorem~\cite[Theorem 1.2]{Bir21}
applied to $(X_0,B'_0+E_{1,0}+E_{2,0})$ over a general point of $Z_2$, we conclude that $Z_2=Z$.
This means that $E_{2,0}$ dominates the base as well.
Let $(F,B'_F+E_{1,F}+E_{2,F})$ be the restriction of 
$(X_0,B'_0+E_{1,0}+E_{2,0})$ to a general fiber.
Note that $\dim F < \dim X$ and $(F,B'_F+ E_{1,F}+E_{2,F})$ is a dlt log Calabi-Yau pair
of regularity zero.

By Lemma~\ref{lem:curves-generating}, we can further assume that $C_{1,0}$ and $C_{2,0}$ satisfy:
\begin{equation}\label{eq:2-equiv}
K_{X_0}\cdot C_{i,0} = -2, \quad
E_{1,0}\cdot C_{i,0} = 1, \quad
E_{2,0}\cdot C_{i,0} = 1, \text{ and } \quad 
B'_0 \cdot C_{i,0}=0.
\end{equation} 
Since $\rho(X_0/Z)=2$ and the curves
$C_{1,0}$ and $C_{2,0}$ are linearly independent in $N_1(X_0)$, we conclude that 
\[
-K_{X_0} \sim_{\qq,Z} 2E_{1,0} \sim_{\qq,Z} 2E_{2,0}
\text{ and }
B'_0 \sim_{\qq,Z} 0.
\]
We claim that $-K_{X_0}$ is nef over the base.
Let $C$ be an effective curve in $X_0$ which is contracted to a point.
Assume that $-K_{X_0}\cdot C <0$, then by the $\qq$-linear equivalence~\eqref{eq:2-equiv}, we conclude that 
\[
E_{1,0}\cdot C < 0 
\text{ and }
E_{2,0}\cdot C<0.
\]
Hence, $C$ must be contained in both $E_{1,0}$ and $E_{2,0}$, leading to a contradiction.
Since $X_0\rightarrow Z$ is a Fano type morphism,
it is a relative Mori dream space~\cite[Corollary 1.3.1]{BCHM10}.
Hence, $-K_{X_0}$ is semiample and big over $Z$.
Then, the following conditions hold for a general fiber:
\begin{itemize}
    \item $B'_{F_0}=0$,
    \item $-K_{F_0}$ is semiample and big, and 
    \item $-K_{F_0} \sim_\qq 2E_{1,F}\sim_\qq 2E_{2,F}$.
\end{itemize}
Since $\dim F_0 <\dim X$, by induction on the dimension, we can apply Lemma~\ref{lem:semiample-big} to $F$ to conclude that $F_0 \simeq \pp^1$.
This leads to a contradiction, since in this case we would have $C_{1,0}=C_{2,0}$.
We conclude that the cone~\eqref{eq:cone-1}
has a unique $(K_X+B')$-negative extremal ray.
\end{proof}

\section{Minimal log discrepancies of regularity one} 

In this section, we prove the 
theorem regarding the minimal log discrepancies of klt singularities of regularity one.
Theorem~\ref{thm:reg-one-coeff} is a generalization of Theorem~\ref{introthm:ACC} for log pairs of regularity one.
The main ingredients in the proof of the theorem 
are the existence of bounded complements~\cite[Theorem 1.8]{Bir19},
the existence of dlt modifications~\cite[Theorem 3.1]{KK10}, 
and the existence of the curves produced in Section~\ref{sec:cone-nef}.
Through the proof of Theorem~\ref{thm:reg-one-coeff}, we will often use the following lemma to control the generators of the relative cone of curves.

\begin{lemma}
\label{lem:curves-generating}
Let $\phi \colon Y\rightarrow Y_0$
and $\phi_0 \colon Y_0\rightarrow X$ be two projective birational maps.
Assume both $\phi$ and $\phi_0$ are relative Mori dream spaces.
Let $C_1,\dots, C_s$ be a set of curves in $Y_0$ that generate 
$N_1(Y_0/X)$
and let $C'_{s+1},\dots,C'_{r}$ be a set of curves in $Y$ that generate
$N_1(Y/Y_0)$.
Assume that $\phi({\rm Ex}(\phi))$ does not contain any of the $C_i$'s.
For each $i \in \{1,\dots,s\}$, let $C_i'$ be the strict transform of $C_i$.
Then, we have that
\[
\langle C'_1,\dots,C'_r\rangle = N_1(Y/X).
\]
\end{lemma}

\begin{proof}
Let $D$ be a $\qq$-divisor on $Y$ that intersect each of the $C_i'$ trivially.
It suffices to show that $D$ is numericallly trivial over $X$.
Since $D$ intersects $C'_{s+1},\dots, C'_{r}$ trivially
and these curves generate $N_1(Y/Y_0)$, 
we conclude that $D$ is numerically trivial over $Y_0$.
Since $\phi$ is a relative Mori dream space, 
we conclude that $D$ is $\qq$-linearly trivial over $Y_0$.
Write $D=\phi^* D_0$ for some $\qq$-divisor $D_0$ on $Y_0$.
For each $i\in \{1,\dots,s\}$, we have that
\[
0=D\cdot C'_i = \phi^*D_0 \cdot C'_i = D_0 \cdot C_i. 
\]
We conclude that $D_0$ intersects each $C_i$, with $i\in \{1,\dots,s\}$, trivially.
Hence, $D_0$ is numerically trivial over $X$
and then it is $\qq$-linearly trivial over $X$.
Thus, we have that $D$ is numerically trivial over $X$. This concludes the proof.
\end{proof}

Now, we turn to prove the ascending chain condition for minimal log discrepancies of regularity one near zero.
The proof will be divided into four steps depending on the structure of the dual complex of the modification.
The proof of each of these cases is similar in flavor.
Some crucial details differ, so we give a comprehensive proof in each case, despite some repetition.

\begin{theorem}\label{thm:reg-one-coeff}
Let $n$ be a positive integer.
Let $\Lambda\subset \qq$ be a set satisfying the descending chain condition with rational accumulation points.
There exists a constant $N:=N(n,\Lambda)$, 
only depending on $n$ and $\Lambda$,
satisfying the following.
Let
\[
\mathcal{M}_{n,\Lambda,r}:=
\{
{\rm mld}(X,\Delta;x) \mid 
\text{
$(X,\Delta;x)$ has regularity $r$ and ${\rm coeff}(\Delta)\subset \Lambda$
}
\}.
\]
Then, the set 
\begin{equation}\label{eq:mld-int-1/N}
\mathcal{M}_{n,\Lambda,1}\cap \left(0,\frac{1}{N}\right)
\end{equation} 
satisfies the ascending chain condition.
\end{theorem}

\begin{proof}
By Lemma~\ref{lem:reg-bounded-comp}, there exists a constant $N:=N(n,\Lambda)$, only depending on $n$ and $\Lambda$, satisfying the following.
For each $(X,\Delta;x)$ as in the statement, there exists a $N$-complement $(X,B;x)$ so that
\[
{\rm reg}(X,\Delta;x) =
{\rm reg}(X,B;x).
\]
Furthermore, by Lemma~\ref{lem:reg-comp-lcc}, we know that $x$ is a log canonical center of $(X,B;x)$.
Note that every divisor computing a log discrepancy in
$(0,\frac{1}{N})$ of $(X,\Delta;x)$ must be a log canonical place of $(X,B;x)$.
Indeed, the log discrepancies of $(X,B;x)$ belong to the set
\[
\zz_{\geq 0}\left[ \frac{1}{N}\right],
\]
and we have that $a_E(X,\Delta)\leq a_E(X,B)$ for each $E$.
Thus, if $a_E(X,\Delta)<\frac{1}{N}$, then
we have that $a_E(X,B)=0$.
Hence, in order to prove the ascending chain condition of the set~\eqref{eq:mld-int-1/N}, it suffices to show that the minimum of the log discrepancies of $(X,\Delta;x)$ at the log canonical places of $(X,B;x)$ satisfies the ascending chain condition.

Let $(Y,B_Y)$ be a dlt modification of $(X,B;x)$ (see Lemma~\ref{lem:existence-dlt-mod}).
We may assume that $\lfloor B_Y\rfloor$ has at least one component mapping onto $x$.
Shrinking around $x$, we may assume that every log canonical center
of $(X,B;x)$ passes through $x$.
Let $E_1,\dots,E_r$ be the prime components of $\lfloor B_Y\rfloor$.
By Lemma~\ref{lem:class-dual-comp}, we know that $\mathcal{D}(Y,B_Y)$
belong to four different classes.
We will prove the theorem in each of these cases.\\

\textit{Step 1:} We prove the theorem in the case that $\mathcal{D}(Y,B_Y)$ is a closed interval: $E_1$ maps onto a log canonical center $Z_1\supsetneq x$,
    $E_r$ maps onto a log canonical center $Z_r\supsetneq x$, 
    $Z_1\neq Z_r$, and each $E_i$ with $i\in \{2,\dots,r-1\}$ maps onto $x$.\\
    
By Lemma~\ref{lem:not-point}, we know that $a_{E_1}(X,\Delta)$ 
and $a_{E_r}(X,\Delta)$ belong to a set $\mathcal{M}(n,N,\Lambda)$, only depending on $n,N$ and $\Lambda$, satisfying the ascending chain condition.
We let 
\[
\mathcal{C}(n,N,\Lambda):=
\{ 1-m \mid m\in \mathcal{M}(n,N,\Lambda)\}.
\] 
By~\cite[Theorem 1]{Mor20}, we can find a projective birational morphism $\phi_0\colon Y_0\rightarrow X$ which extracts the divisors $E_1$ and $E_r$.
By abuse of notation, we denote by $E_1$ and $E_r$ the strict transforms of these divisors on $Y_0$.
Write 
\[
\phi_0^*(K_X+B)=K_{Y_0}+E_1+E_r+B_{Y_0}'.
\]
Note that $E_1$ and $E_r$ intersect non-trivially in $Y_0$.
By Corollary~\ref{cor:two-comp-toric}, we conclude that $Z_{1,r}:=E_1\cap E_r$ is an irreducible variety and every log canonical center of $(Y_0,E_1+E_r+B_{Y_0}')$ can be extracted by a toroidal blow-up at the generic point of $Z_{1,r}$.
We can write 
\[
\phi_0^*(K_X+\Delta)=
K_{Y_0}+\Delta_{Y_0}=
K_{Y_0}+c_1E_1+c_rE_r+\Delta_{Y_0}',
\]
where $c_1,c_r\in \mathcal{C}(n,N,\Lambda)$.
In particular, the coefficients of the sub-pair $(Y_0,\Delta_{Y_0})$ belong to the set
\[
\Lambda \cup \mathcal{C}(n,N,\Lambda).
\]
Furthermore, the minimal log discrepancy of $(Y_0,\Delta_{Y_0})$ can be extracted by a toroidal blow-up at the generic point of $Z_{1,r}$.
Thus, the minimal log discrepancy is computed by a surface minimal log discrepancy.
Hence, the statement follows from Lemma~\ref{lem:k-th-surf-toric-mld}.\\

\textit{Step 2:} We prove the statement of the theorem in the case that $\mathcal{D}(Y,B_Y)$ is a circle: $E_1$ maps onto a log canonical center $Z_1 \supsetneq x$ and every other $E_i$ maps onto $x$.\\ 

Let $E_i$ be a prime component which maps onto $x$.
By~\cite[Theorem 1]{Mor20}, we can find two projective birational morphisms
$\phi_0\colon Y_0\rightarrow X$
and 
$\phi_1\colon Y_1\rightarrow Y_0$, so that
$\phi_0$ only extracts $E_1$
and $\phi_1$ only extracts $E_i$.
By abuse of notation, we may denote both the divisor and the strict transform by $E_i$.
In the model $Y_1$, the divisors $E_i$ must intersect $E_1$ in two disjoint irreducible components.
By Lemma~\ref{lem:not-point}, we know that $a_{E_1}(X,\Delta)$ belongs to the set $\mathcal{M}(n,N,\Lambda)$.
Let $C_i$ be a general $\pp^1$ in $E_i$ when considering its log pair structure with regularity zero (see Definition~\ref{def:gen-p1}).
Note that the curve $C_i$ generates $N_1(Y_1/Y_0)$.
We can obtain a dlt modification of
$(Y_0,E_1+E_i+B_{Y_0}')$ by performing toroidal blow-ups at the generic points of the intersection $E_i\cap E_1$.
We may replace $(Y,B_Y)$ with this new dlt modification.
Each prime component 
$E_2,\dots, E_r$ of $\lfloor B_Y\rfloor$ which is exceptional over $Y_0$ admits the structure of a log pair of regularity zero.
For each $i\in \{2,\dots,r\}$, we denote by $C_i$ a general $\pp^1$ in $E_i$.
By Lemma~\ref{lem:curves-generating}, we have that 
\[
\langle C_2,\dots,C_r\rangle = 
N_1(Y/Y_0).
\] 
Indeed, the curve $C_i$ generate $N_1(Y_1/Y_0)$ 
and the curves $C_j$, with $j\in \{1,\dots,r\}\setminus \{i\}$,
generate $N_1(Y/Y_1)$.
Let $\phi\colon Y\rightarrow Y_0$ be the induced projective birational morphism.
We can write 
\begin{equation}\label{eq:pb-delta}
\phi^*(K_{Y_0}+\Delta_{Y_0}+c_1E_1)= K_Y+ c_1E_1+(1-\alpha_2)E_2 +\dots + (1-\alpha_r)E_r + \Delta_{Y}', 
\end{equation} 
and 
\begin{equation}\label{eq:pb-b}
\phi^*(K_{Y_0}+B_{Y_0}) =
K_Y+E_1+\dots+E_r+B_{Y}'.
\end{equation} 
Subtracting~\eqref{eq:pb-delta} and~\eqref{eq:pb-b}, we obtain:
\begin{equation}\label{eq:lin-eq1}
(1-c_1)E_1 + \alpha_2E_2 +\dots+\alpha_r E_r \sim_{\qq, X} \Delta_{Y}' - B_{Y}'.
\end{equation} 
Observe that every curve $C_i$ is disjoint from both $\Delta_{Y}'$ and $B_{Y}'$.
Note that each $C_i$ lies in the smooth locus of $E_i$ and
intersect exactly two other of the prime divisors $E_1,\dots,E_r$.
In particular, the intersection $E_i\cdot C_i$ is a negative integer $-m_i$.
Intersecting the $\qq$-linear equation~\eqref{eq:lin-eq1} with each of the $C_i$, with $i\in \{2,\dots,r\}$, leads to a system of linear equations:
\begin{equation}
\label{equation:lin-sys-1}
\left[ 
\begin{matrix}
m_2 & -1 & 0 & \dots & 0 \\
-1 & m_3 & -1 & \dots & 0\\
0 & -1 & m_4 & \dots & 0 \\
\vdots & \vdots &\vdots &\vdots &\vdots \\
0 & 0 & 0 & -1 & m_r
\end{matrix}
\right] 
\left[ 
\begin{matrix}
\alpha_2 \\
\alpha_3 \\
\alpha_4 \\
\vdots \\
\alpha_r 
\end{matrix}
\right]
=
\left[ 
\begin{matrix} 
1-c_1 \\
0 \\
\vdots \\
0 \\
1-c_1
\end{matrix}
\right]
\end{equation} 
Note that the matrix of equation~\eqref{equation:lin-sys-1} is of full rank as the curves $C_2,\dots,C_r$ span the relative cone of curves over $Y_0$.
Let $X(\sigma)$ be the toric singularity associated to the continued fraction $[m_1,\dots,m_r]$.
Let $x_0\in X(\sigma)$ be the torus invariant point.
Let $(X(\sigma),c_1T_1+c_1T_2)$ be the log pair structure obtained by considering both torus invariant divisors with coefficient $c_1$.
By construction, we have that
\[
{\rm mld}(X,\Delta;x) = \min\{\alpha_2,\dots,\alpha_r\} =
{\rm mld}(X(\sigma),c_1T_1+c_1T_2;x_0).
\]
By Lemma~\ref{lem:k-th-surf-toric-mld}, we conclude that the minimum of the $\alpha_i$, with $i\in \{2,\dots, r\}$, belongs to a set satisfying the ascending chain condition which only depends on $n,N$ and $\Lambda$.\\

\textit{Step 3:} We prove the statement of the theorem in the case that $\mathcal{D}(Y,B_Y)$ is a closed interval: $E_1$ maps onto a log canonical center $Z_1\supsetneq x$, and each $E_i$ with $i\geq 2$ maps onto $x$.\\

By~\cite[Theorem 1]{Mor20}, there exists a projective birational morphism
$\phi_0 \colon Y_0\rightarrow X$ whose exceptional locus is purely divisorial and consists on the divisor $E_1$.
By the same result, we can find a projective birational morphism $\phi\colon Y_1\rightarrow Y_0$
which extracts the divisor $E_r$.
In the model $Y_1$ the divisors $E_1$ and $E_r$ intersect at an irreducible set of codimension two.
Furthermore, the center of $E_i$, with $i\in \{2,\dots,r-1\}$ maps to such irreducible subvariety.
By Theorem~\ref{thm:comp}, the singularity of $(Y_1,E_1+E_r+B_{Y_1}')$ is toroidal at the generic point of $E_1\cap E_r$.
Then, each $E_i$ with $i\in \{2,\dots,r-1\}$ can be extracted by a sequence of toroidal blow-ups.
By Lemma~\ref{lem:not-point}, we know that $a_{E_1}(X,\Delta)$ is contained in the set $\mathcal{M}(n,N,\Lambda)$
which satisfies the ascending chain condition.
We let $c_1=1-a_{E_1}(X,\Delta)$.
Let $(Y_1,B_{Y_1})$ be the log pull-back of $(X,B)$ to $Y_1$.
Let $(E_r,B_{E_r})$ be the pair obtained by adjunction of 
$(Y_1,B_{Y_1})$ to $E_r$.
By Lemma~\ref{lem:from-reg-1-to-0}, we conclude that $(E_r,B_{E_r})$ is a pair
of regularity zero.
Furthermore, since $E_r$ corresponds to one of the end-points of $\mathcal{D}(Y,B_Y)$,
we conclude that $(E_r,B_{E_r})$ has a unique divisorial log canonical center.
We call such divisor $S_r$ and we denote $B'_{E_r}:=B_{E_r}-S_r$.
By Theorem~\ref{introthm:nef-cone}, we can find a movable curve $C_r$ on $E_r$ satisfying the following conditions:
\begin{itemize}
    \item $C_r$ is in the smooth locus of $E_r$, 
    \item $C_r$ intersects $S_r$ transversally in at most $k$ points, and
    \item $C_r$ is either disjoint from $B'_{E_r}$ or it intersects $B'_{E_r}$ transversally in at most $k$ points.
\end{itemize}
Here, $k$ is a constant which only depends on $n$ and $N$.
We can obtain a dlt modification of $(Y_1,E_1+E_r+B_{Y_1}')$ by 
performing a sequence of toroidal blow-ups at the generic point of $E_1\cap E_r$.
We replace $(Y,B_Y)$ by this new dlt modification.
For $i\in \{2,\dots,r-1\}$, each prime component $E_i$ of $\lfloor B_Y\rfloor$ admits the structure of a log pair
of regularity zero.
For each $i\in \{2,\dots,r-1\}$, we denote by $C_i$ a general $\pp^1$ in $E_i$ (see Definition~\ref{def:gen-p1}).
We may identify $C_r$ with its strict transform in $Y$.
By Lemma~\ref{lem:curves-generating}, we have that
\[
\langle C_2,\dots, C_r\rangle = N_1(Y/Y_0).
\]
Let $\phi \colon Y\rightarrow Y_0$ be the induced birational morphism.
We can write
\begin{equation}\label{eq:s3-1}
\phi^*(K_{Y_0}+\Delta_{Y_0}+c_1E_1) =
K_Y + c_1E_1 + (1-\alpha_2) E_2 + \dots + (1-\alpha_r)E_r+ \Delta'_Y
\end{equation} 
and
\begin{equation}\label{eq:s3-2}
\phi^*(K_{Y_0}+B_{Y_0}) = K_Y+E_1+\dots+E_r+B'_Y.
\end{equation}
Here, $B'_Y\geq \Delta'_Y$. Subtracting~\eqref{eq:s3-1} and~\eqref{eq:s3-2}, we obtain that
\begin{equation}\label{eq:s3-3} 
(1-c_1)E_1 + \alpha_2 E_1 +\dots + \alpha_r E_r \sim_{\qq,X} \Delta_Y'-B_Y'.
\end{equation} 
Observe that for each $i\in \{2,\dots, r-1\}$, the curve $C_i$ lies in the smooth locus
of $Y$ and $C_i$ intersects exactly two other prime components of $\lfloor B_Y\rfloor$.
Moreover, each such $C_i$ is disjoint from $B_Y'$ and hence from $\Delta_Y'$.
On the other hand, we have that
\[
E_{r-1}\cdot C_r = c_{r-1,0} \leq k, \quad 
E_r\cdot C_r = -\frac{m_r}{N!}, \text{ and }
(\Delta'_Y-B'_Y) \cdot C_r = -\beta_{r,0},
\] 
where $c_{r-1,0}$ and $m_r$ are positive integers and
$\beta_{r,0}$ belongs to a set satisfying the ascending chain condition.
The denominator of the intersection $E_r\cdot C_r$ is bounded 
as the variety $Y$ has quotient singularities of order at most $N$ along $C_r$.
We let $c_{r-1}:=c_{r-1,0}N!$ and $\beta_r:=\beta_{r,0}N!$.
Intersecting the $\qq$-linear equivalence~\eqref{eq:s3-3} with the curves
$C_i$ with $i\in \{2,\dots,r\}$, we obtain the following system of linear equations:
\[
\left[ 
\begin{matrix}
m_2 & -1 & 0 & \dots & \dots & 0 \\
-1 & m_3 & -1 & \dots & \dots & 0\\
0 & -1 & m_4 & \ddots & \dots & 0 \\
\vdots & \vdots &\vdots &\vdots &\ddots &\vdots \\
0 & 0 & 0 & -1 & m_{r-1} & -1 \\
0 & 0 & 0 & 0 & -c_{r-1} & m_r
\end{matrix}
\right] 
\left[ 
\begin{matrix}
\alpha_2 \\
\alpha_3 \\
\alpha_4 \\
\vdots \\
\alpha_{r-1} \\
\alpha_r 
\end{matrix}
\right]
=
\left[ 
\begin{matrix} 
1-c_1 \\
0 \\
0 \\
\vdots \\
0 \\
\beta_r
\end{matrix}
\right]
\]
Let $\sigma\subset \qq^2$ be the cone of the surface toric singularity
corresponding to the continued fraction $[m_2,\dots,m_{r-1}]$.
Let $v_1$ and $v_2$ be the lattice generators of $\sigma$.
Let $x_1,\dots, x_r$ be a regular decomposition of $\sigma$.
Define $y_2:=m_rx_r - c_{r-1}x_{r-1}$. 
Let $\tau$ be the cone spanned by $x_r$ and $y_2$.
Then, $\tau$ admits a regular decomposition
given by lattice vectors $w_1,\dots, w_s$, where $s\leq c_{r-1}$.
There exists a unique linear function $M$ on $\qq^2$ for which
\[
M(v_1)=1-c_1 \text{ and } M(y_2)=\beta_r.
\]
Let $\Sigma$ be the cone spanned by $v_1$ and $y_2$.
Note that the values $1-c_1$ and $\beta_r$ belong to a set,
which only depends on $n,N$, and $\Lambda$,
and satisfies the ascending chain condition.
By construction, this linear function takes value $\alpha_i$ at $x_i$, for every $i\in \{2,\dots,r\}$.
If the minimizer of $M$ in ${\rm relint}(\Sigma)\cap \zz^2$ is contained in 
${\rm relint}(\sigma)\cap \zz^2$, then we are done by Lemma~\ref{lem:k-th-surf-toric-mld}.
Otherwise, we can assume that the minimizer is contained in ${\rm relint}(\tau)$.
By Lemma~\ref{lem:min-lattice}, we know that the minimizer must be one of the lattice elements
$\{w_1,\dots,w_s\}$.
Assume that it is attained at $w_s$.
By Lemma~\ref{lem:k-th-surf-toric-mld}, the value of $M$ at $w_s$ belongs to a set
satisfying the ascending chain condition which only depends on $\beta_r$ and $1-c_1$.
We can replace $y_2$ with $w_s$ and proceed inductively.
Since $s\leq c_{r-1}$, this argument can be repeated at most $c_r$ times 
until the minimizer is contained in ${\rm relint}(\sigma)$.
Hence, one of the $(c_{r-1}+1)$ smaller values of $M$ in ${\rm relint}(\Sigma)\cap \zz^2$ must
lie in ${\rm relint}(\sigma)$. 
Then, the statement follows from Lemma~\ref{lem:k-th-surf-toric-mld}.\\

\textit{Step 4:} We prove the statement of the theorem in the case that $\mathcal{D}(Y,B_Y)$ is a closed interval: each $E_i$ with $i\in \{1,\dots, r\}$ maps onto $x$.\\

By~\cite[Theorem 1]{Mor20}, there exists a projective birational morphism
$\phi_0 \colon Y_0\rightarrow X$ whose exceptional locus is purely divisorial and consists of the divisors $E_1$ and $E_r$.
In the model $Y_0$, the divisors $E_1$ and $E_r$ intersect at an irreducible set of codimension two.
Furthermore, the center of $E_i$, with $i\in \{2,\dots,r-1\}$ maps onto such irreducible subvariety.
By~\cite[Theorem 18.22]{Kol92}, the singularity of $(Y_0,E_1+E_r+B_{Y_0}')$ is toric at the generic point of $E_1\cap E_r$.
Then, each $E_i$ with $i\in \{2,\dots,r-1\}$ can be extracted by a sequence of toroidal blow-ups.
For $i\in \{1,r\}$, we denote by $(E_i,B_{E_i})$ the pair obtained by adjunction of 
$(Y_0,E_1+E_r+B_{Y_0}')$ to $E_i$.
By Lemma~\ref{lem:from-reg-1-to-0}, we know that $(E_i,B_{E_i})$ has regularity zero
and $\lfloor B_{E_i}\rfloor$ is a prime divisor $S_i$ which corresponds to the intersection $E_1\cap E_r$.
Proceeding as in Step 3, for $i\in \{1,r\}$, we can find a movable curve $C_i$ on $E_i$ satisfying the following conditions:
\begin{itemize}
    \item $C_i$ is in the smooth locus of $E_i$, 
    \item $C_i$ intersects $S_i$ transversally in at most $k$ points, and
    \item $C_i$ is either disjoint from $B'_{E_i}$ or it intersects $B'_{E_i}$ transversally in at most $k$ points.
\end{itemize}
Here, $k$ is a constant which only depends on $n$ and $N$.
We claim that the curves $C_1$ and $C_r$ span $N_1(Y_0/X)$.
Note that $N_1(Y_0/X)$ is two-dimensional, so it suffices to prove that $C_1$ and $C_r$ are linearly independent.
Assume that $C_1+\alpha C_r \equiv_X 0$  for some $\alpha\in \rr$.
Observe that $E_1\cdot C_1$ is negative and $E_1\cdot C_r$ is positive.
We conclude that $\alpha$ must be positive.
Hence, we have an effective curve which is numerically trivial over $X$.
This gives a contradiction.
Thus, $\langle C_1,C_r\rangle = N_1(Y_0/X)$.

We can produce a dlt modification of $(Y_0,E_1+E_r+B_{Y_0}')$ by 
performing a sequence of toroidal blow-ups at the generic point of $E_1\cap E_r$.
We replace $(Y,B_Y)$ by this dlt modification.
For $i\in \{2,\dots,r-1\}$, each prime component $E_i$ of $\lfloor B_Y\rfloor$ admits the structure of a log pair
of regularity zero.
For each $i\in \{2,\dots,r-1\}$, we denote by $C_i$ a general $\pp^1$ in $E_i$.
We may identify both $C_1$ and $C_r$ with their strict transforms on $Y$.
By Lemma~\ref{lem:curves-generating}, we have that
\[
\langle C_1, C_2,\dots, C_r\rangle = N_1(Y/X).
\]
Let $\phi \colon Y\rightarrow X$ be the induced birational morphism.
We can write
\begin{equation}\label{eq:s4-1}
\phi^*(K_X+\Delta) =
K_Y + (1-\alpha_1) E_1 + \dots + (1-\alpha_r)E_r + \Delta'_Y
\end{equation} 
and
\begin{equation}\label{eq:s4-2}
\phi^*(K_X+B_X) = K_Y+E_1+\dots+E_r+B'_Y,
\end{equation}
where $B'_Y\geq \Delta'_Y$. 
Subtracting the equations~\eqref{eq:s4-1} and~\eqref{eq:s4-2}, we obtain that
\begin{equation}\label{eq:s4-3} 
\alpha E_1 + \alpha_2 E_1 + \dots + \alpha_r E_r \sim_{\qq,X} \Delta_Y'-B_Y'.
\end{equation} 
Observe that for each $i\in \{2,\dots, r-1\}$, the curve $C_i$ lies in the smooth locus
of $Y$ and $C_i$ intersects exactly two other prime components of $\lfloor B_Y\rfloor$.
Moreover, each such $C_i$ is disjoint from $B_Y'$ and hence from $\Delta_Y'$.
On the other hand, for $i\in \{1,r\}$, we have that
\[
E_i \cdot C_i = - \frac{m_i}{N!}, \quad
(\Delta'_Y-B'_Y)\cdot C_i = -\beta_{i,0}, \quad
E_2 \cdot C_1= c_{2,0}, \text{ and  } \quad
E_{r-1}\cdot C_r= c_{r-1,0}.
\] 
where $c_{r-1,0},c_{2,0},m_1$, and $m_r$ are positive integers.
$\beta_{1,0}$ and $\beta_{r,0}$ are rational numbers which belongs 
to a set satisfying the ascending chain condition.
The denominators of the intersection products $E_1\cdot C_1$ and $E_r\cdot C_r$ are bounded above
as the variety $Y$ has quotient singularities of order at most $N$ along $C_1$ and $C_r$.
We set
\[
c_{r-1}=c_{r-1,0}N!, \quad  c_2:=c_{2,0}N!, \quad \beta_1=\beta_{1,0}N!,
\text{ and  } \quad \beta_r:=\beta_{r,0}N!.
\]
Intersecting the $\qq$-linear equivalence~\eqref{eq:s4-3}, with the curves
$C_i$ with $i\in \{2,\dots,r\}$, we obtain the following system of linear equations:
\begin{equation}
\label{eq:s4-4}
\left[ 
\begin{matrix}
m_1 & -c_2 & 0 & \dots & \dots & 0 \\
-1 & m_2 & -1 & \dots & \dots & 0\\
0 & -1 & m_3 & \ddots & \dots & 0 \\
\vdots & \vdots &\vdots &\vdots &\vdots &\vdots \\
0 & 0 & 0 & -1 & m_{r-1} & -1 \\
0 & 0 & 0 & 0 & -c_{r-1} & m_r
\end{matrix}
\right] 
\left[ 
\begin{matrix}
\alpha_1 \\
\alpha_2 \\
\alpha_3 \\
\vdots \\
\alpha_{r-1} \\
\alpha_r 
\end{matrix}
\right]
=
\left[ 
\begin{matrix} 
\beta_1 \\
0 \\
0 \\
\vdots \\
0 \\
\beta_r
\end{matrix}
\right]
\end{equation} 
Note that this linear system~\eqref{eq:s4-4} has a unique solution as the matrix on the left-hand side has full rank.
This follows from the fact that the $C_i$'s generate the relative cone of curves.

Let $\sigma\subset \qq^2$ be the cone of the surface toric singularity
corresponding to the continued fraction $[m_2,\dots,m_{r-1}]$.
Let $v_1$ and $v_2$ be the lattice generators of $\sigma$.
Let $x_1,\dots, x_r$ be a regular decomposition of $\sigma$.
Define
\[
y_1:=m_1x_1 - c_2x_2 \text{ and }
y_2:=m_rx_r - c_{r-1}x_{r-1}.
\]
Let $\tau_1$ be the cone spanned by $v_1$ and $x_1$, and
let $\tau_2$ be the cone spanned by $x_r$ and $y_2$.
Then, $\tau_1$ (resp. $\tau_2)$ admits a regular decomposition
given by lattice vectors $u_1,\dots, u_{s_1}$ (resp. $w_1,\dots,w_{s_2}$) where $s_1,s_2\leq c_{r-1}$.
There exists a unique linear function $M$ on $\qq^2$ for which
\[
M(y_1)=\beta_1 \text{ and } M(y_2)=\beta_r.
\]
Let $\Sigma$ be the cone spanned by $y_1$ and $y_2$.
Note that the values $\beta_1$ and $\beta_r$ belong to a set,
which only depends on $n,N$, and $\Lambda$,
and satisfies the ascending chain condition.
By construction, we have that
\[
M(x_i)=\alpha_i \text{ for every $i\in \{1,\dots,r\}$.}
\]
If the minimizer of $M$ in ${\rm relint}(\Sigma)\cap \zz^2$ is contained in 
${\rm relint}(\sigma)\cap \zz^2$, then we are done by Lemma~\ref{lem:k-th-surf-toric-mld}.
Otherwise, we can assume that the minimizer is contained in ${\rm relint}(\tau_1)$.
By Lemma~\ref{lem:min-lattice}, we know that the minimizer must be one of the lattice elements
$\{u_1,\dots,u_{s_1}\}$.
Assume that it is attained at $u_1$.
By Lemma~\ref{lem:k-th-surf-toric-mld}, 
the value $M(u_1)$
belongs to a set
satisfying the ascending chain condition which only depends 
on $\beta_1$ and $\beta_r$.
We can replace $y_1$ with $u_1$ and proceed inductively.
Since $s_1,s_r\leq kN!$ this argument can be repeated at most $kN!$ times 
until the minimizer is contained in ${\rm relint}(\sigma)$.
Hence, one of the $(kN!+1)$ smaller values of $M$ in ${\rm relint}(\Sigma)\cap \zz^2$ must
lie in ${\rm relint}(\sigma)$. 
Then, the statement follows from Lemma~\ref{lem:k-th-surf-toric-mld}.
\end{proof}

\begin{proof}[Proof of Corollary~\ref{introcor:bounded-ext}]
Let $(X;x)$ be a $n$-dimensional $\qq$-factorial klt singularity of regularity one.
By Lemma~\ref{lem:reg-bounded-comp} and Lemma~\ref{lem:reg-comp-lcc}, 
there exists a constant $N:=N(n)$, only depending on $n$, satisfying the following.
There exists a $N$-complement $(X,B;x)$ so that
\[
{\rm reg}(X,\Delta;x)={\rm reg}(X,B;x).
\]
Moreover, $x$ is a log canonical center of $(X,B;x)$.
By the proof of Theorem~\ref{thm:reg-one-coeff}, we know that the minimal log discrepancy
at the divisors over $x$ which are log canonical places of $(X,B;x)$
satisfies the ascending chain condition.
In particular, there exists a constant $a(n)$, only depending on $n$, so that
$a_E(X,\Delta)<a(n)$ for some log canonical place $E$ of $(X,B;x)$ which maps onto $x$.
Then, the existence of the projective birational morphism
$Y\rightarrow X$ follows from~\cite[Theorem 1]{Mor20}.
\end{proof}

\begin{proof}[Proof of Corollary~\ref{introthm:ACC}]
In the proof of Theorem~\ref{thm:reg-one-coeff}, if we start with a finite set $\Lambda$, 
then the log discrepancy of every log canonical center of $(X,B;x)$ which does not map to $x$ belongs 
to a finite set due to Lemma~\ref{lem:not-point}.
This implies that the coefficients of the surface toric germ constructed in the proof of Theorem~\ref{thm:reg-one-coeff}
belong to a finite set.
Hence, the corollary follows given that the accumulation points of the log discrepancies 
of toric surface singularities with coefficients in a finite set can only accumulate to zero.
\end{proof}

\bibliographystyle{habbrv}
\bibliography{bib}

\end{document}